\documentclass[11pt,reqno]{amsart} 

\usepackage[letterpaper,left=1.5in,right=1.5in]{geometry}

\usepackage{amsfonts}
\usepackage{amssymb}
\usepackage[mathscr]{eucal}

\DeclareMathAlphabet{\mathpzc}{OT1}{pzc}{m}{it}

\usepackage{mathtools} 
\usepackage{yhmath} 
\usepackage{shuffle}
\usepackage{extarrows}

\usepackage{etoolbox}

\usepackage{color}
\usepackage[width=.75\textwidth]{caption}

\usepackage{subcaption}

\usepackage{tikz}
\usetikzlibrary{cd,patterns}

\usepackage{graphicx,multicol}

\usepackage[%
colorlinks,
linkcolor=blue!80!black, 
citecolor=green!80!black, 
urlcolor=magenta!80!black,
]{hyperref}

\usepackage{algorithm}
\usepackage[noEnd=True]{algpseudocodex}
\usepackage{pifont}
\newcommand{\cmark}{\text{\ding{51}}}
\newcommand{\xmark}{\text{\ding{55}}}

\usepackage[backend=biber,maxbibnames=99,sorting=nyt]{biblatex}
\newtheorem{theorem}{Theorem}
\newtheorem{corollary}{Corollary}
\newtheorem{lemma}{Lemma}
\newtheorem{proposition}{Proposition}

\theoremstyle{definition}
\newtheorem{definition}{Definition}
\newtheorem{remark}{Remark}
\newtheorem{example}{Example}

\definecolor{myred}{rgb}{.7,.1,.1}
\definecolor{myblue}{rgb}{.1,.1,.7}
\definecolor{mygreen}{rgb}{.1,.6,.1}
\definecolor{mygray}{rgb}{0.5,0.5,0.5}
\definecolor{mymauve}{rgb}{0.58,0,0.82}
\definecolor{lablue}{rgb}{0,0.9,0.93}

\usepackage[textwidth=1in,textsize=normalsize]{todonotes} 

\newcommand{\demph}[1]{\textcolor{myblue}{\emph{#1}}}

\def\qand{\quad\hbox{and}\quad}

\mathchardef\mhyphen="2D 

\def\End{\operatorname{\mathrm{End}}}

\def\id{\mathrm{id}}
\def\tw{\mathrm{tw}}

\DeclareMathOperator{\Char}{char}

\newcommand\dhxrightarrow[2][]{%
  \mathrel{\ooalign{$\xrightarrow[#1\mkern4mu]{#2\mkern4mu}$\cr%
  \hidewidth$\rightarrow\mkern4mu$}}
}

\def\port{\mathsf{port}}

\def\into{\hookrightarrow}
\def\onto{\twoheadrightarrow}

\def\park{\mathsf{park}}

\def\splus{\mathbin{\boldsymbol{+}}}
\def\stimes{\mathbin{\boldsymbol{\times}}} 
\def\sdot{\mathbin{\boldsymbol{\cdot}}} 
	
\def\scirc{\mathbin{\boldsymbol{\circ}}} 

\def\dototimes{\mathrel{\dot{\otimes}}}

\def\Spec{\mathsf{Sp}}
\def\Set{\mathsf{Set}}
\def\Vec{\mathsf{Vec}}

\def\bimon{\mathsf{Bimon}}

\def\gvec{\mathsf{gVec}}

\def\sym{\mathrm{Sym}}
\def\qsym{\mathrm{QSym}}
\def\nsym{\mathrm{NSym}}
\def\ncsym{\mathrm{NCSym}}
\def\ncqsym{\mathrm{NCQSym}}
\def\fqsym{\mathrm{FQSym}}
\def\ssym{\boldsymbol{\mathfrak{S}\mathrm{Sym}}}
\def\sqsym{\mathfrak{S}\mathrm{QSym}}
\newcommand{\rqsym}[1][r]{{}^{#1}\qsym}
\newcommand{\rncqsym}[1][r]{{}^{#1}\ncqsym}
\def\pqsym{\mathrm{PQSym}}
\def\pfsym{\mathrm{PFSym}}
\def\cqsym{\mathrm{CQSym}}

\def\K{\mathcal K}
\def\Kbar{\overline{\mathcal K}}
\def\Kvee{{\mathcal K}^\vee}
\def\Kbv{\overline{\mathcal K}^\vee}

\def\bfSigma{\mathbf{\Sigma}}
\def\bfPi{\mathbf{\Pi}}

\def\bfE{\mathbf E}
\def\bfL{\mathbf L}
\def\bij{\mathbf{bij}}
\def\cyc{\mathbf{cyc}}

\def\pf{\mathbf{PF}}

\def\rf{\mathbf{F}}
\def\prf{\vec{\mathbf{F}}}

\def\bfc{\mathbf c}
\def\bfx{\mathbf x}

\def\b{\mathbf b} 
\def\d{\mathbf d} 
\def\h{\mathbf h}
\def\p{\mathbf p}
\def\q{\mathbf q}
\def\a{\mathbf a}

\def\rfp{\mathrm p}

\def\rfE{\mathrm E}

\def\ZZ{\mathbb Z}
\def\QQ{\mathbb Q}

\newcommand{\rspec}[2][r]{\prescript{\,#1\,}{}{#2}}

\newcommand{\upspec}[2][r]{{#2}_{\underbracket[.6pt][.25ex]{\scriptstyle\hspace{.15em}#1\hspace{.05em}}}}
\newcommand{\downspec}[2][r]{{#2}_{\wideparen{\scriptstyle\hspace{.1em}#1\hspace{.1em}}}}

\newcommand{\rc}[1][r]{\prescript{#1}{}{\bfc}}

\newcommand{\rd}[1][r]{\prescript{#1}{}{\d}}

\def\cfT{\mathcal{T}}
\def\cfS{\mathcal{S}}

\begin{document}

\title[Interpolation in Species]{Interpolation in Species and a Lift of the Hopf Algebra of $r$-Quasisymmetric Functions}

\author{Aaron Lauve}
\address{Department of Mathematics and Statistics, Loyola University Chicago -- Chicago, IL, USA}
\email{alauve@luc.edu} 


\author{Anthony Lazzeroni}
\address{Department of Mathematics, North Park University -- Chicago, IL, USA}
\email{aalazzeroni@northpark.edu}

\date{March 19, 2026}

\begin{abstract}%
In this manuscript we lift the theory of $r$-quasisymmetric functions to the theory of Hopf monoids. We provide a general method of interpolating between two Hopf monoids, one being the free monoid on a positive comonoid and the other being the free commutative monoid on a positive comonoid.
\end{abstract}


\maketitle



\section{Introduction}

The Hopf algebras $\left\{\rqsym\right\}_{r\in\ZZ_+}$ form a one-parameter family that interpolates between the Hopf algebra $\qsym$ of quasisymmetric functions (at $r=1$) and the Hopf algebra $\sym$ of symmetric functions (at $r=\infty$). That is, there are injections of Hopf algebras
\[
\sym \into \cdots \rqsym[r+1] \into \rqsym \into \cdots \qsym
\quad(\forall r=2,3,\ldots).
\]
 for this interpolating family. For fixed $r$ and $n$, Hivert \cite{hivert2004local}  defines $\rqsym_n$ as the invariant space of an ``$r$-local'' action of the symmetric group $\mathfrak S_n$ on the polynomial ring $\QQ[\bfx_n]$ in commuting variables $\bfx_n = (x_1, x_2, \ldots, x_n)$. Specifically, for adjacent transposition $\sigma_i = (i,i{+}1)$, he defines
\begin{equation*} 
    \sigma_i\ast_r (x_i^a x_{i+1}^b):=
    \begin{cases} 
    x_{i+1}^a x_{i}^b & \text{if }a<r \text{ or } b<r\\
    x_i^a x_{i+1}^b & \text{else}
    \end{cases}.
\end{equation*}
He checks that this action satisfies the braid relations, giving a bonafide action of the symmetric group on $\QQ[\bfx_n]$. This action is not multiplicative; that is, $\sigma \ast_r (fg) \neq (\sigma \ast_rf)(\sigma \ast_rg)$ for general polynomials $f, g$. 
Nevertheless, the space is shown to be closed under ordinary polynomial product. A coproduct is defined via a standard alphabet-doubling trick; and it is shown to be compatible with polynomial product as $n$ tends to infinity. So $\rqsym = \lim_{n\to\infty} \rqsym_n$ is afforded the structure of graded connected Hopf algebra. 
A variant $\rncqsym$ built from noncommuting variables $\bfx$ was also introduced, interpolating between $\ncqsym$ and $\ncsym$.

In \cite{hivert2004local} one also finds a conjecture, later proven by Garsia and Wallach in \cite{garsia2007rqsym}, that $\rqsym_n$ is a free module over $\sym_n$ with dimension $n!$. These interpolating families have recently regained attention: \cite{lazzeroni2023powersum} introduces a powersum basis that interpolates between natural bases for $\sym$ and $\qsym$; and \cite{aliniaeifard2024generalized} poses $\rncqsym$ as a natural setting for generating functions for certain digraph colorings.

The present work was motivated by a desire to unify the construction of these two families of Hopf algebras. The formalism of Hopf monoids in the monoidal category 
$(\Spec,\sdot)$ of species seemed the right place to start our investigation for a few reasons: (i) the study of combinatorial structures under break and join operations (as promoted by Rota and collaborators \cite{joni1979coalgebras}) 
easily carries over to this category; (ii) after the work of Joyal \cite{joyal1981species}, Bergeron--Labelle--Leroux \cite{bergeron1998combinatorial} and others, there are easy ways ($+$, $\times$, $\sdot$, and $\circ$) to build  
new species of structures from old; and (iii) in \cite{aguiar2010monoidal} one finds a systematic study of species and their functors to the category $(\gvec, \otimes)$, where the Hopf algebras $\sym$, $\ncsym$ and the like make their home. As a bonus, (iv) it also happens that (co-, bi-)monoid axioms are easier to verify for species than for graded vector spaces. (See Proposition 8.35 and Remarks 8.8 and  8.36 in \cite{aguiar2010monoidal} for one perspective on this phenomenon.)

\subsubsection*{Contents}

In Section \ref{sec:prelims} we give further background on $(\Spec, \sdot)$, and assorted Hopf monoids therein. Of particular importance are the species $\bfL$ of linear orders and the exponential species $\bfE$. We recount some universal constructions defined using these Hopf monoids.
In Section \ref{sec:interpolation}, we make precise what we mean by an interpolating family $\bigl\{\rc\bigr\}_{r\in\ZZ_+}$ of Hopf monoids between two given ones, $\b$ and $\d$. Theorem \ref{th:abelianization} gives an interpolating family of Hopf monoids between $\bfL\scirc\p$ and $\bfE\scirc\q$ starting from any surjective morphism $\theta :\p \onto \q$ of positive comonoids in $(\Spec,\sdot)$.
We employ the machinery of Fock functors from $(\Spec, \sdot)$ to $(\Vec, \otimes)$ in Section \ref{sec:fock-functors}, capturing the motivating examples $\rqsym$ and $\rncqsym$ among many others. (A Hopf monoid of parking functions is introduced to reach some of these.) We close with a number of related questions in Section \ref{sec:further-questions}. 

\subsection*{Acknowledgements}
The authors thank Marcelo Aguiar and Rafael Gonz\'alez D'Le\'on for several helpful conversations throughout the writing of this manuscript.

\section{Vector Species, and Hopf Monoids Therein}
\label{sec:prelims}

We briefly catalog what is needed from the theory of species and Hopf monoids. We follow closely the notation in \cite{aguiar2010monoidal}, where further proofs and all details may be found.

\subsection{Species} 
\label{sec:species}
Let $\Set$ (respectively, $\Vec$) be the category of sets (vector spaces over $\Bbbk$), with morphisms being arbitrary set maps (vector space maps). Let $\Set^{\times}$ be the category of finite sets with bijections as morphisms. A \demph{set species} is a functor $\rfp:\Set^{\times} \to \Set$ (assigning a set $\rfp[I]$ for each $|I|<\infty$, and a bijection $\rfp[\sigma]:\rfp[I] \to \rfp[J]$ for each bijection $\sigma:I\to J$). A \demph{(vector) species} $\p:\Set^{\times} \to \Vec$ is defined similarly. 

Call $\p$ \demph{connected} if $\p[\emptyset]\cong\Bbbk$ and \demph{positive} if $\p[\emptyset]=\{0\}$.  
Call $\p$ \demph{linearized} if there is a set species $\rfp$ so that $\p[I] = \Bbbk\rfp[I]$ for all $I$ and $\p[\sigma] = \Bbbk\rfp[\sigma]$ for all $\sigma:I\to J$. 
We consider here only \demph{finite} species: $\dim\p[I]<\infty$ for all $|I|<\infty$.
We also assume $\Char \Bbbk=0$.

\begin{example}\label{ex:elem examples}
Three elementary species, before we continue our discussion.
The \demph{exponential species} $\mathbf E$, taking $\mathbf E[I]$ to be the $1$-dimensional space with basis $\bigl\{\ast_I\bigr\}$ for all finite sets $I$.
The \demph{trivial species} $\mathbf 1$, concentrated in degree 0: $\mathbf 1[I] = \Bbbk$ if $I=\emptyset$; and $\{0\}$ otherwise. 
The \demph{linear orders} $\bfL$, taking $\bfL[I]$ to be the space $\Bbbk\{l \mid l\text{ a linear (total) ordering of }I\}$.
\end{example}

Below and throughout: $S \sqcup T=I$ indicates an ordered decomposition of a set $I$ into (possibly empty) subsets; while $X\vdash I$ indicates an ordinary, unordered partition of $I$.

\begin{definition}\label{def:operations}
The \demph{sum}, 
\demph{Cauchy product}, and \demph{substitution product} of two species $\p$ and $\q$ are defined, respectively, by
\begin{gather*}
(\p\splus\q)[I] := \p[I] \oplus \q[I],\;\  
(\p\sdot\q)[I] := \bigoplus_{S\sqcup T = I} \p[S]\otimes \q[T],\\ \text{and }\;
(\p\scirc\q_+)[I] := \bigoplus_{X \vdash I} \p[X] \otimes \Bigl(\bigotimes_{A\in X} \q[A]\Bigr),
\end{gather*}
where $(\p\scirc\q_+)[\emptyset]$ is defined to be $\p[\emptyset]$.
The Cauchy product of two species $\p,\q$ is viewed as the space of pairs of elements, {\it i.e.}, $(p,q)\in \p[S]\otimes\q[T]\subseteq (\p\sdot\q)[I]$.
\end{definition}

For $n\geq 0$, let $\p_n$ be the new species concentrated in cardinality $n$,
\[
\p_n[I] = 
\begin{cases}
    \p[I] & \text{if } |I|=n, \\
    \{0\} & \text{otherwise}.
\end{cases}
\]

\begin{example}\label{ex:more examples}
The preceding constructions allow us to recast $\mathbf1$ and $\bfL$ as follows: 
\[
    \mathbf1 = \bfE_0
    \qand
    \bfL = {\bf1} \splus \bfE_1 \splus \bfE_1^{\sdot2} \splus \bfE_1^{\sdot3} \splus \cdots. 
\]
Here are some additional species that will be important in what follows.
\begin{itemize}\setlength{\itemsep}{0ex}

\item Any species is built from its restrictions via infinite sum. Or, viewed the other way around, every species $\p$ has a \demph{canonical decomposition} $\p = \sum_{n\geq0} \p_n$ \cite[Ex. 1.3.6]{bergeron1998combinatorial}. We will need three related constructions:
\[
    \p_{+}:=\sum_{1\leq n}\p_n,
    \quad
    \p_{<r}:=\sum_{0\leq n<r}\p_n,
    \qand
    \p_{\geq r}:=\sum_{r\leq n}\p_n.
\]
\item The species of \demph{set partitions} $\bfPi$ is $\bfE\circ\bfE_+$, while the species of \demph{ordered set partitions} (or, set compositions) $\bfSigma$ is $\bfL\circ\bfE_+$.
\end{itemize}
\end{example}

\subsubsection{Enumeration}
\label{sec:enum}

Following \cite{bergeron1998combinatorial}, we denote the \demph{exponential generating series} of species as $\q(x) = \sum q_n\frac{x^n}{n!}$ where $q_n=\dim \q[n]$, and we denote the \demph{ordinary generating series} of species as $\widetilde{\q}(x) = \sum \widetilde{q}_nx^n$ where $\widetilde{q}_n = \dim \q[n]_{\mathfrak S_n}$. (If $\q$ is linearized, these count labelled and unlabelled structures, respectively.) It is clear that the exponential generating series of $\bf1$, $\bfE_k$, and $\bfE$ are 
    \[
    \mathbf{1}(x) = 1, \qquad \bfE_k(x) = x^k/k!, \qquad \bfE(x) = e^x,
    \]
while their ordinary generating series are
    \[
    \widetilde{\mathbf{1}}(x) = 1, \qquad \widetilde{\bfE_k}(x) = x^k, \qquad \widetilde{\bfE}(x) = \frac{1}{1-x}.
    \]

The fundamental operations on species translates nicely to generating series:
\[
\begin{array}{rlcrl}
    (\p\splus\q)(x) =& \p(x) + \q(x) & & 
    \widetilde{(\p\splus\q)}(x)= & \widetilde{\p}
    (x) + \widetilde{\q}(x) \\
    (\p\sdot\q)(x) =& \p(x)\q(x) & & 
    \widetilde{(\p\sdot \q)}(x) =& \widetilde{\p}
    (x)\widetilde{\q}(x) \\
    (\p\scirc\q)(x) =& \p(\q(x)), & & 
\end{array}
\]
where the last row only holds if $\q[\emptyset]=\emptyset$. (See \cite[Theorem 2.15]{bergeron2008introduction} for the missing formula, written in terms of cycle index series.) 

\begin{example} Returning to the species in Example \ref{ex:more examples}, we have
    \[
     \bfE_+(x) = e^x-1, \qquad \bfL(x) = \frac{1}{1-x}, \qquad \bfPi(x) = e^{e^x-1}, \qquad
     \bfSigma(x) = \frac{1}{2-e^x},
    \]
and
    \[
    \widetilde{\bfE_+}(x) = \frac{x}{1-x}, \qquad \widetilde{\bfL}(x) = \frac{1}{1-x}, \qquad \widetilde{\bfPi}(x) = \prod_{k=1}^\infty \frac{1}{1-x^k}, \qquad
     \widetilde{\bfSigma}(x) = \frac{1-x}{1-2x}.
    \]
\end{example}

\subsection{Monoidal structures}
\label{sec:monoidal structures}

The category $\Spec$ is a braided monoidal category under Cauchy product, with $\mathbf1$ as the unit object (and trivial brading, $u\otimes v \mapsto v\otimes u$). 
 
Let $(\mathsf{C},\bullet)$ be a monoidal category with unit element $\mathrm{U}$. Recall a \demph{monoid} in $\mathsf{C}$ is an object $\mathrm B$ equipped with morphisms $\mu:\mathrm B\bullet \mathrm B \to \mathrm B$ and $\iota:\mathrm{U}\to \mathrm B$ satisfying certain associativity and unital axioms.
Similarly, $\mathrm B$ is a \demph{comonoid} in $\mathsf{C}$ if it's equipped with morphisms $\Delta:\mathrm B \to \mathrm B \bullet \mathrm B$ and $\epsilon:\mathrm B \to \mathrm{U}$ satisfying certain dual axioms.
Finally, $\mathrm B$ is a \demph{bimonoid} in (a braided monoidal category) $\mathsf{C}$ if these maps are compatible (briefly, $\Delta$ and $\epsilon$ are monoid-morphisms); and this bimonoid $\mathrm{B}$ is a \demph{Hopf monoid} in $\mathsf{C}$ if the identity morphism $\id$ on $\mathrm B$ is invertible in the convolution monoid, $(\End(\mathrm{B}), \boldsymbol{\ast})$. (Here, the convolution $f\mathrel{\boldsymbol{\ast}}g$ of two morphisms is defined as the composition $\mu \circ (f\bullet g)\circ \Delta$.) If $\id^{-1}$ exists, we call it the \demph{antipode}---ordinarily denoted ``$\mathrm{s}$.''

Below, we deal exclusively with connected species, so most  axioms involving the (co)units are trivial. We display the remaining axioms in our context.

\medskip\noindent
\emph{Associativity axiom for monoids.} For all decompositions $S\sqcup T$ of a finite set $I$, there exist morphisms $\mu_{S,T}:\b[S]\otimes \b[T] \to \b[I]$ making the following diagram commute:
\begin{equation}
\label{eq:associative}
\begin{tikzcd}
    \b[R]\otimes\b[S]\otimes\b[T] \arrow[rr, "\id_R\otimes \mu_{S,T}"]\arrow[d, "\mu_{R,S}\otimes \id_T"'] & & \b[R]\otimes\b[S\sqcup T] \arrow[d, "\mu_{R,S\sqcup T}"]\\
    \b[R\sqcup S]\otimes\b[T] \arrow[rr, "\mu_{R\sqcup S,T}"'] & & \b[R\sqcup S\sqcup T]
\end{tikzcd}
\end{equation}

\medskip\noindent
\emph{Coassociativity axiom for comonoids.} For all decompositions $S\sqcup T$ of a finite set $I$, there exist morphisms $\Delta_{I,J}: \b[I] \to \b[S]\otimes \b[T]$ making the following diagram commute:

\begin{equation}
\label{eq:coassociative}
\begin{tikzcd}
    \b[R\sqcup S\sqcup T] \arrow[rr, "\Delta_{R,S\sqcup T}"]\arrow[d, "\Delta_{R\sqcup S,T}"'] & & \b[R]\otimes\b[S\sqcup T] \arrow[d, "\id_S\otimes \Delta_{S,T}"]\\
    \b[R\sqcup S]\otimes\b[T] \arrow[rr, "\Delta_{R,S}\otimes \id_T"'] & & \b[R]\otimes\b[S]\otimes\b[T]
\end{tikzcd}
\end{equation}

\medskip\noindent
\emph{Compatibility axiom for binomoids.} 
Let $\tw:V\otimes W \to W\otimes V$ be the standard braiding on vector spaces over $\Bbbk$. Given two decompositions $S\sqcup T$ and $S'\sqcup T'$ of a finite set $I$, put 
\begin{gather}
\label{eq:ABCD}
	A=S\cap S', \ \ B = S\cap T', \ \ C = T\cap S', \qand D = T\cap T'. 
\end{gather}
Then the following diagram commutes:
\begin{equation}
\label{eq:compatible-delta}
\begin{tikzcd}
    \b[A]\otimes\b[B]\otimes\b[C]\otimes\b[D] \arrow[rr, "\id_A\otimes \tw \otimes \id_D"] & & \b[A]\otimes\b[C]\otimes\b[B]\otimes\b[D] \arrow[d, "\mu_{A,C}\otimes \mu_{B,D}"]\\
    \b[S]\otimes\b[T] \arrow[r, "\mu_{S,T}"'] \arrow[u, "\Delta_{A,B}\otimes\Delta_{C,D}"] & \b[I] \arrow[r,"\Delta_{S',T'}"'] & \b[S']\otimes\b[T'] 
\end{tikzcd}
\quad\mbox{\ }   
\end{equation}

\smallskip\noindent
\emph{(Left-)Convolution-invertibility of $\id$.} Note that the identity morphism in $\End(\b)$ under convolution is $\iota\circ\epsilon$. So, for all finite sets $I$, we require morphisms $\mathrm{s}_I:\b[I] \to \b[I]$ making the following diagram (and its ``right invertible'' counterpart) commute:
\begin{equation}
\label{eq:antipode}
\begin{tikzcd}[column sep=tiny]
    {\displaystyle\bigoplus_{S\sqcup T}}\,\b[S]\!\otimes\!\b[T] \arrow[rr, "\mathrm{s}_S\otimes \id_T"] & & {\displaystyle\bigoplus_{S\sqcup T}}\, \b[S]\!\otimes\!\b[T] \arrow[d, start anchor={[yshift=1.5ex]south}, yshift=.25ex, "\oplus\mu_{S,T}"]\\
    \b[I] \arrow[r, "\epsilon_I"'] \arrow[u, end anchor={[yshift=1.5ex]south}, yshift=.25ex, "\oplus\Delta_{S,T}"'] & \mathbf1[I] \arrow[r,"\iota_I"'] & \b[I]\,.
\end{tikzcd}
\end{equation}

\begin{remark}
\label{rem:hopf-emptyset}
    (a). Note \eqref{eq:antipode} differs from \eqref{eq:associative}, \eqref{eq:coassociative}, and \eqref{eq:compatible-delta} in that we cannot compare maps one decomposition of $I$ at a time. Things aren't as bad as they seem: after \cite[Prop. 8.10]{aguiar2010monoidal}, $\b$ is a Hopf monoid if and only if $\b[\emptyset]$ is a Hopf algebra. This means it suffices to check \eqref{eq:antipode} in the case $I=\emptyset$. 

 (b).
If $\b$ is connected, then $\b[\emptyset]=\Bbbk$ is trivially a Hopf algebra. (So we make no further mention of antipodes in what follows.)
\end{remark}

\subsubsection{Additional constructions}
\label{sec:constructions}
We also need the following for species $\b$ and $\d$:
\begin{itemize}\setlength{\itemsep}{0ex}
    \item If $\b$ and $\d$ are Hopf, then so is $\b\sdot\d$. 
    \item If $\mathcal I\subseteq \b$ is a bi-ideal of a connected bimonoid $\b$, then $\b/\mathcal I$ is a Hopf monoid.
    \item if $\b\xhookrightarrow{\pi}\d$ is a morphism of Hopf monoids, then so is $\d^* \dhxrightarrow{\pi^*}\b^*$. (In particular, the dual species $\b^*:=\bigoplus_{I}\b[I]^*$ is Hopf whenever $\b$ is.)
\end{itemize}

\subsection{The functors \texorpdfstring{$\cfT(\mhyphen)$}{T(-)} and \texorpdfstring{$\cfS(\mhyphen)$}{S(-)}}
\label{sec:T and S}

We consider the free Hopf monoid $\cfT(\p) = \bfL\scirc\p$, and free commutative Hopf monoid $\cfS(\p)=\bfE\scirc\p$, built on a positive comonoid $\p$. 
As free objects, they satisfy certain universal properties. We give the first here.

\begin{theorem}[{\cite[Thm. 11.10]{aguiar2010monoidal}}]
\label{th:T(p)}
	Let $\h$ be a Hopf monoid, $\p$ a positive comonoid, and $\zeta:\p\to\h_+$ a morphism of positive comonoids. Let $\eta(\p)$ be the natural inclusion of $\p$ into $\cfT(\p)_+$. There exists a unique morphism of Hopf monoids $\hat{\zeta}:\cfT(\p)\to\h$ such that
	\[
	\begin{tikzcd}
		\cfT(\p)_+ \arrow[rr, "\hat{\zeta}_{+}"]  &  & \h_+ \\
		& \p \arrow[ul, "\eta(\p)"] \arrow[ur, "\zeta"']& 
	\end{tikzcd}
	\]
	is a commutative diagram of positive comonoids. 
\end{theorem}
To describe the Hopf monoid $\cfT(\p)$, we need some notation.
From Definition \ref{def:operations}, we might write an element of $\cfT(\p)$ as $l\otimes \bigotimes_{A\in X}p_{A}$, for a given set partition $X\vdash I$ and given elements $l\in \bfL[X]$ and $\{p_A \in \p[A] \mid A\in X\}$. But if $l$ is a linear order on $X$, we use the simpler notation ``$(p_A)_{A\in F}$'' where $F\vDash I$ is a \emph{set composition} of $I$. We write simple tensors in $\cfS(\p)$ as $\{p_{A}\}_{A\in X}$.

\begin{example}\label{ex:T(p)} 
    (\emph{The free Hopf monoid $\cfT(\p)$ on $\p$.}) Let $S\sqcup T = I$ be a decomposition, with $F\vDash S$ and $G\vDash T$. The product maps are simple enough:
\begin{align*}
    \mu_{S,T}:\cfT(\p)[S]&\otimes \cfT(\p)[T] \to \cfT(\p)[I] 
    & {(p_A)}_{A\in F} &\otimes {(p_{B})}_{B\in G} \mapsto {(p_C)}_{C\in F|G},
\end{align*}
where $F|G$ is the concatenation of two compositions.
For details on the coproduct, see ``dequasi-shuffling'' in \cite[Sec. 11.2.4]{aguiar2010monoidal}.
Briefly, we first remark that any composition $H\vDash I$ induces a pair of set compositions $F\vDash S$ and $G\vDash T$, with any given block $C\in H$ decomposing as $A\sqcup B$ for $A\in F$, $B\in G$. We use a modified Sweedler notation for $\Delta^{\p}$: given $p_C\in\p[C]$ with $A\sqcup B = C$, we write
$\Delta_{A,B}^{\p}(p_C) = \sum_{(p_C)} p_A\otimes p_B$. Extending this to the multi-tensor product $(p_c)_{C\in H}$, we define the coproduct in $\cfT(\p)$ as follows. 
\begin{align*}
\Delta_{S,T} : \cfT(\p)[I] &\to \cfT(\p)[S] \otimes \cfT(\p)[T] &
    {(p_C)}_{C\in H} &\mapsto \!\!\sum_{((p_C)_{C\in H})}\!\! {(p_{A})}_{A\in F} \otimes {(p_B)}_{B\in G}.
\end{align*}
Edge cases (either $A=S\cap C$ or $B=T\cap C$ is empty) are handled in the obvious way.\footnote{Since $\p$ is positive, we have no definition for the expression, {\it e.g.}, $\Delta^{\p}_{\emptyset,C}$, appearing implicitly in the definition of $\Delta_{S,T}$. In such cases, omit the corresponding $\Delta^{\p}$ and give $p_C$, unaltered, to $F$ or $G$ as appropriate.}
\end{example}

The Hopf monoid $\cfS(\p)$ is defined similarly (replace all set compositions appearing above with set partitions).

\begin{example} The Hopf monoids $\bfSigma$ and $\bfPi$ of set compositions and set partitions, respectively, are relevant examples. We have $\bfSigma = \cfT(\bfE_+)$ and $\bfPi = \cfS(\bfE_+)$. For both, we take the simplest (linearized) comonoid structure for $\bfE_+$:
\[
    \bfE_+[I] \xrightarrow{\Delta_{S,T}} \bfE_+[S]\otimes \bfE_+[T]
    \quad\hbox{with}\quad
    \ast_I \mapsto \ast_S \otimes \ast_T
\]
for all nonempty decompositions $I=S\sqcup T$. Continuing the example for $\bfPi=\cfS(\bfE_+)$, the coproduct applied to a set partition $\phi$ may be written
\[
    \phi \mapsto \phi\vert_{S} \otimes \phi\vert_{T},
\]
where, {\it e.g.}, $\phi\vert_{S}$ is the partition of $S$ built by intersecting each block of $\phi$ with $S$ (discarding empty intersections). The product on set partitions is union of blocks: 
\[
    \phi \otimes \psi \mapsto \phi \sqcup \psi.
\]
\end{example}

\begin{remark}\label{rem:duality}
Let $\p$ be a positive monoid, then $\p^*$ is naturally a positive comonoid, and we may consider the Hopf monoid $\cfT(\p^*)$ as in Section \ref{sec:T and S}. 
We define the functor $\cfT^{\vee}$ via 
\[
	\cfT^{\vee}(\p) = (\cfT(\p^*))^*.
\]
Likewise for $\cfS^{\vee}$. The Hopf monoids $\cfT^{\vee}(\p)$ and $\cfS^{\vee}(\p)$ satisfy their own universal properties,\footnote{See \cite[\S\S 11.4 \& 11.5]{aguiar2010monoidal} for details.}
but this fact won't be needed in what follows. 
We mention them only because the Hopf algebras $\qsym$ and $\ncqsym$ from the introduction (revisited in Section \ref{sec:fock-functors}) are more naturally described via fock-functor images of $\cfT^{\vee}(\bfE_+)$.%
\end{remark}

\section{Interpolation in $(\Spec,\sdot)$ }
\label{sec:interpolation}

Fix two species $\b$ and $\d$ and a totally-ordered set $\mathcal N$. We say the (one-parameter) family of species $\{\rc\}_{r \in \mathcal N}$ \demph{interpolates between $\b$ and $\d$} if there exist surjective morphisms making the following diagram commute
\begin{gather}\label{eq:ports-from-b}
\begin{tikzcd}[ampersand replacement=\&]
    \b \arrow[rr, pos=.6, "\port_{r}"'] \arrow[bend left=23, pos=.4, "\port_s"]{rrrr}{} 
    \&  \& \rc \arrow[rr,"\port^r_{s}"'] \arrow[bend left=23,shorten >=1.5mm, pos=.6, "\port^r"]{rrrr}{} \& \& \rspec[s]{\bfc} \arrow[rr, pos=.35, "\port^{s}"'] \&  \& \d \ .
\end{tikzcd}
\end{gather}    
And additionally, $\port^r_t = {\port}^s_t \circ {\port}^r_s$ ($\forall\,r\leq s<t$).

\begin{remark}
\label{rem:interpolation-variants}
We focus our discussion, and all proofs that follow, on the case where $\mathcal N = \ZZ_{\geq1}$, with $\b = \rc[1]$ and $\d = \rc[\infty] = \lim_{r\to\infty} \rc$. (The latter notion is defined carefully below.) 
But in examples, it may be convenient to allow our indexing to proceed in the opposite direction (say, $\b = \rc[\infty]$ and $\rc[1] = \d$). 

More importantly, we extend the definition of interpolating family (as in the introduction) to allow for a chain of injections in place of surjections. One can pass between the two notions by taking duals of the associated species, {\it cf.} Section \ref{sec:constructions}.
\end{remark}

Recall from \cite[Def.~1.21]{bergeron1998combinatorial} that a sequence $(\rspec{\p})_{r\geq1}$ of species \demph{converges to} $\q$ up to isomorphism (written $\lim_{r\to\infty} \rspec{\p}\cong\q$) if for $n\geq0$, there exists an $r_0$ such that $\rspec{\p}_{\leq n} \cong \q_{\leq n}$ for all $r> r_0$.
A \demph{decreasing filtration} of a species $\b$, denoted $\b=\bigcup_r\upspec{\b}$, is a sequence of species satisfying $\b = \upspec[1]{\b} \supseteq \upspec[2]{\b} \supseteq \cdots \supseteq \mathbf1$.
An \demph{increasing filtration} of a species $\d$, denoted $\d=\bigcup_r\downspec{\d}$, has $\mathbf1 \subseteq \downspec{\d}$, $\downspec[r-1]{\d}\subseteq \downspec{\d}$ and $\lim_{r\to\infty}\downspec{\d}\cong\d$. 

\begin{definition}\label{def:interpolating-family}
Suppose $\pi: \b \onto \d$ is a morphism of Hopf monoids and there are filtrations $\bigcup_r\upspec{\b}$ and $\bigcup_r\downspec{\d}$. We say $\rc := \upspec{\b}\sdot\downspec{\d}$ $(r\geq 1)$ is an \demph{interpolating family of Hopf monoids} if each $\rc$ is Hopf and the commuting diagram \eqref{eq:ports-from-b} comprises Hopf maps that factor $\pi$.
\end{definition}

\subsection{An axiomatic approach}
\label{sec:Hopf-filtrations}

Let $\b=\bigcup_r\downspec{\b}$ and $\d=\bigcup_r\downspec{\d}$ be increasing filtrations (as defined above) of sub-comonoids and Hopf submonoids, respectively. Put $\downspec{\b} = \mathbf1 \splus \downspec{\b}^+$ and write $\langle \downspec{\b}^+\rangle$ for the ideal in $\b$ generated by $\downspec{\b}^+$. Note that $\langle \downspec{\b}^+\rangle \subseteq \langle \downspec[s]{\b}^+\rangle$ for all $r<s$.

We now give axioms guaranteeing the existence of an interpolating family of Hopf monoids.

\begin{enumerate}
\item \label{ax:comm} $\pi:\b\to\d$ is a surjective Hopf monoid map with $\d$ commutative; 
\item \label{ax:filt} $\b$ and $\d$ have filtrations as above, and $\pi$ preserves these, {\it i.e.}, $\pi(\downspec{\b}) \subseteq \downspec{\d}$ for all $r$;
\item \label{ax:comp} species complements $\b=\upspec{\b} \splus \langle \downspec{\b}^+\rangle$ are defined so that $\upspec{\b} \supseteq \upspec[s]{\b}\supseteq \mathbf1$ for all $r<s$; 
\item \label{ax:stab} the antipode stabilizes the comonoid $\downspec{\b}$ for all $r$.\footnote{As we have taken our Hopf monoids to be connected, this condition may be dropped. Extending the present work to the non-connected case, this axiom is most easily achieved by requiring that the $\b$-filtration (like the $\d$-filtration) is one of Hopf monoids.}
\end{enumerate}

Under the hypotheses on $\b$ and $\downspec{\d}$, the Cauchy product $\b\sdot\downspec{\d}$ is a Hopf monoid, with ``coordinatewise'' product and coproduct by \cite[Props. 3.50, 3.75, 3.76]{aguiar2010monoidal}. 
The following Lemma realizes $\rc$ as a quotient of $\b\sdot\downspec{\d}$ by a Hopf ideal.

\begin{lemma}
Given $\b$ and $\d$ satisfying Axioms \ref{ax:comm}--\ref{ax:stab}. Let $\rspec{\mathcal I}$ be the ideal in $\b$ generated by the elements
\begin{gather*}
\bigl\{ b\otimes 1_{\d} - 1_\b\otimes \pi(b) \ \mid \ b\in\downspec{\b}^+ \bigr\}.
\end{gather*}
Then $(\b\sdot\downspec{\d})/\rspec{\mathcal I}$ is isomorphic to $\upspec{\b}\sdot \downspec{\d}$ as species. Additionally, $\rspec{\mathcal I}$ is a Hopf ideal.
\end{lemma}

\begin{proof}
Clearly, $\b\sdot \downspec{\d} \equiv \upspec{\b}\sdot\downspec{\d} \mod \rspec{\mathcal I}$. Towards the first claim, a straightforward degree argument shows that a basis for $\upspec{\b}\sdot\downspec{\d}$ remains linearly independent in the quotient.

\bigskip
Towards the second claim, let $x=b\otimes 1_{\d} - 1_\b\otimes \pi(b)$ be a generator of $\rspec{\mathcal I}$ (so $b\in\downspec{\b}^+$). We need to show that the antipode $S=S^{\b}\otimes S^{\d}$ satisfies $S(x)\in \rspec{\mathcal I}$ and the coproduct $\Delta_{S,T}$ satisfies
\begin{gather} \tag{$\ast$} \label{eq:in-ideal} 
\Delta_{S,T}(x) \ \in \ (\b\sdot\downspec{\d})[S]\dototimes \rspec{\mathcal I}[T] \ + \ \rspec{\mathcal I}[S]\dototimes(\b\sdot\downspec{\d})[T].
\end{gather}
Axiom \ref{ax:stab} gets us most of the way through the first point. (It's a simple matter to show that $S^{\b}(b)\in \downspec{\b}$ implies $S(x) \in \rspec{\mathcal I}$.) For the second, a standard ``adding zero'' trick and appeal to Axiom \ref{ax:filt} will suffice. 
(We use the modified Sweedler notation $\Delta^{\b}_{S,T}(b) = \sum_{(b)} b_S \otimes b_T$ for the coproduct in $\b$ with $S,T\neq\emptyset$.)
 \begin{align*}
\Delta_{S,T}(x) &= \sum_{(b)} b_S\otimes 1 \dototimes b_T\otimes 1 - \sum_{(\pi(b))} 1\otimes \pi(b)_S \dototimes 1\otimes \pi(b)_T \\
=& \sum_{(b)} b_S\otimes 1 \dototimes b_T\otimes 1 
\mp b_S\otimes 1 \dototimes 1\otimes \pi(b_T) - 
1\otimes \pi(b_S) \dototimes 1\otimes \pi(b_T), \\
\intertext{since $\pi$ is a comonoid map. Regrouping,}
=& \sum_{(b)} b_S\otimes 1 \dototimes \Bigl[b_T\otimes 1 - 1\otimes \pi(b_T)\Bigr] + 
\sum_{(b)} \Bigl[b_S\otimes 1 - 1\otimes \pi(b_S)\Bigr]\dototimes 1\otimes \pi(b_T).
\end{align*}
Since $\downspec{\b}$ is a comonoid, each $b_S,b_T$ appearing above belongs to $\downspec{\b}^+$ whenever $S,T$ are nonempty. So \eqref{eq:in-ideal} is satisfied. Analysis of the edge cases ({\it e.g.}, when $S\sqcup T = \emptyset\sqcup I$) proceeds similarly. Conclude $\rspec{\mathcal I}$ is a Hopf ideal.
\end{proof}

\begin{remark}
The commutative assumption on $\d$ is essential, since $(b\otimes 1)\cdot (1\otimes d) = b\otimes d = (1\otimes d)\cdot(b\otimes 1)$. Indeed, assume $b\in\downspec{\b}^+$. Using the first equality, we get $b\otimes d \equiv \pi(b)\cdot d \mod \rspec{\mathcal I}$; using the second, we get $b\otimes d \equiv d\cdot\pi(b) \mod \rspec{\mathcal I}$.
\end{remark}

\begin{definition} \label{def:rc}
Let \demph{$\rc$} be the species quotient $(\b\sdot\downspec{\d})/\rspec{\mathcal I}$. Since this is a quotient of a Hopf monoid by a Hopf ideal, it is also a Hopf monoid.
\end{definition}

We now turn to the maps $\port^\bullet_\bullet$, whose definitions are more intricate.

\begin{lemma}
Let $\varsigma^r:\rc \to \b\sdot\downspec{\d}$ be any (linear) section of the quotient $\vartheta_r: \b\sdot\downspec{\d} \to (\b\sdot\downspec{\d})/\rspec{\mathcal I}$. Letting $\iota$ be the natural inclusion $\b\sdot\downspec{\d} \into \b\sdot\downspec[s]{\d}$ of Hopf monoids for all $r<s$, put $\pi^r_s := \vartheta_s\circ\iota\circ\varsigma^r$. Then $\pi^r_s$ is a morphism of Hopf monoids from $\rc$ to $\rc[s]$.
\end{lemma}

\begin{proof}
The reader can check that $\iota(\rspec{\mathcal I}) \subseteq \rspec[s]{\mathcal I}$ for $r<s$, from which the result will follow.

\smallskip
\noindent\emph{$\pi^r_s$ respects products.}~Given simple tensors $b\otimes d, b'\otimes d' \in \rc$, let the section $\varsigma^r$ yield elements $b\otimes d + f$ and $b'\otimes d + f'$, respectively ($f,f'\in\rspec{\mathcal I}$). The reader can check that $\pi^r_s\circ\mu_{S,T}^{\rc}$ and $\mu^{\rc[s]}_{S,T}\circ(\pi^r_s)^{\otimes2}$ both satisfy
\[
b\otimes d \dototimes b'\otimes d' \ \longmapsto \ (bb')\otimes(dd') \mod \rspec[s]{\mathcal I}
\]
(indicated products taken in $\b$ and $\d$, respectively).

\medskip
\noindent\emph{$\pi^r_s$ respects coproducts.}~Given a simple tensor $b\otimes d \in \upspec{\b}[P]\otimes \downspec{\d}[Q] \subseteq \rc{}[I]$, with section $b\otimes d + f$, write
\[
    \Delta_{S,T}^{\b\sdot\downspec{\d}}(b\otimes d + f) = \sum_{(b),(d)} b_{S\cap P}\otimes d_{S\cap Q} \dototimes b_{T\cap P}\otimes d_{T\cap Q} + \sum_{(f)} f_S\dototimes f_T.
\]
The reader can check that $(\pi^r_s)^{\otimes2}\circ\Delta_{S,T}^{\rc}$ and $\Delta_{S,T}^{\rc[s]}\circ\pi^r_s$ both satisfy
\[
    b\otimes d \ \longmapsto \ \sum_{(b),(d)} b_{S\cap P}\otimes d_{S\cap Q} \dototimes b_{T\cap P}\otimes d_{T\cap Q} \mod \rspec[s]{\mathcal I}\sdot(\b\sdot\downspec[s]{\d}) \splus (\b\sdot\downspec[s]{\d}) \sdot \rspec[s]{\mathcal I}.
\]
We omit the trivially verified (co)unit axioms, so the proof is complete.
\end{proof}

Note that $\b\cong \b\sdot\downspec[1]{\d} \cong \upspec[1]{\b}\sdot\downspec[1]{\d}$ (since $\downspec[1]{\b}^+ = \{0\}$ and $\downspec[1]{\d}=\mathbf1$). So the first map of interest, $\port_1:\b\to \rc[1]$, is just this isomorphism. On the other extreme, note that $\pi^r_{s}$ acts as inclusion on the subspace $\mathbf1\sdot\downspec{\d}$ of $\rc$. We identify these subspaces (and inclusions) with the filtration $\d=\bigcup_r\downspec{\d}$ in what follows.

\begin{definition} \label{def:port}
Define the carrying maps \demph{$\port^\bullet_\bullet$} as follows.
Put $\port_r = \pi_{r}^{r-1}\circ\cdots\circ\pi^2_3\circ\pi^1_2$, $\port^r_s = \pi_{s}^{s-1}\circ\cdots\circ\pi^{r+1}_{r+2}\circ\pi^{r}_{r+1}$, and  $\port^s = \cdots\circ \pi^{s+1}_{s+2}\circ\pi^{s}_{s+1}$. (In the proof of the coming theorem, we argue that: (i) $\port^r_s$ is actually just $\pi^r_s$; and (ii) $\port^s$ is locally finite, and hence well-defined.) These are morphisms of Hopf monoids, as (finite) compositions of such. 
\end{definition}

\begin{theorem} \label{th:axiomatic}
Suppose $\pi:\b\onto\d$ is a surjective morphism of connected Hopf monoids, with $\d$ commutative. Then $\{\rc\}_{r\geq1}$ and the carrying maps $\port^\bullet_\bullet$ from Definitions \ref{def:rc} and \ref{def:port} comprise an interpolating family of Hopf monoids between $\b$ and $\d$. 
\end{theorem}

\begin{proof}
To begin, put $\hat\pi = \cdots \pi^{i}_{i+1} \circ\cdots\circ \pi^2_3\circ\pi^1_2$. Then
\[
    \hat\pi = \port^r\circ\port_r = \port^s\circ\port^r_s\circ\port_r
\]
for all $r<s$. (And, by construction, $\port^r_t = \port^s_t\circ\port^r_s$.) Leaving aside the question of whether or not $\hat\pi$ and $\port^r$ are well-defined, it would seem the diagram in \eqref{eq:ports-from-b} commutes, with all maps being Hopf morphisms factoring $\hat\pi$. Three straightforward checks complete the proof.

\smallskip
\noindent\emph{Claim: $\hat\pi$ and $\port^r$ are well-defined.}~Given the simple tensor $b\otimes d$ in $\rc$, note that $b\in\downspec[s]{\b}$ for some $s>r$. Then $\port^r = \port^s\circ\port^r_s$, with $\port^r_s$ mapping $b\otimes d$ to $1\otimes d'd$ for some $d'\in\downspec[s]{\d}$. (We argue below that $d'=\pi(b)$.) Now all subsequent factors $\pi^i_{i+1}$ of $\port^s$ act as inclusions $\downspec[i]{\d}\into \downspec[i+1]{\d}$. (Again, we identify the subspace $\mathbf1\sdot\downspec[i]{\d}\subseteq\rc[i]$ with $\downspec[i]{\d}$ for all $i$.) So $\port^r$ and $\hat\pi$ are locally finite compositions of Hopf morphisms, hence themselves well-defined Hopf morphisms.

\medskip
\noindent\emph{Claim: $\hat\pi = \pi$.}~Given $1\neq b\in \b$, identify $b$ with $b\otimes 1 \in \rc[1]$. We must show that $\hat\pi(b\otimes 1) = 1\otimes \pi(b)$. As each $\pi^i_{i+1}$ is linear, it suffices to consider $b$ as a product $b_{1}b_{2}\cdots b_{k}$ with each $b_{j} \in \downspec[i_j]{\b}^+ \setminus \downspec[i_j-1]{\b}^+$. For convenience, we take $i_1 < i_2 < \ldots < i_k$. (The fully general case of any choices $i_j\in\ZZ_{>0}$, including repetitions, presents no added technical difficulties, only notational ones.) Then 
\begin{align*}
    \hat\pi(b\otimes 1) &= \port^{i_1}\circ\pi^{i_1-1}_{i_1}\circ\port_{i_1-1}(b\otimes 1) = \port^{i_1}\circ\pi^{i_1-1}_{i_1}(b\otimes 1) \\
    &= \port^{i_1}\bigl(b_{2}\cdots b_{k}\otimes \pi(b_{1})\bigr) \\
    &= \port^{i_2}\circ\pi^{i_2-1}_{i_2}\circ\port_{i_2-1}^{i_1}\bigl(b_{2}\cdots b_{k}\otimes \pi(b_{1})\bigr) \\
    &= \port^{i_2}\bigl(b_{3}\cdots b_{k}\otimes \pi(b_{2})\pi(b_{1})\bigr),
\end{align*}
and so on. Since $\d$ is commutative, the result of $\hat\pi$ on $b\otimes1$ is $1\otimes\prod_{j}\pi(b_{j}) = 1\otimes\pi(b)$, as needed. 

\medskip
\noindent\emph{Claim: $\port^r_s = \pi^r_s$.}~Appealing to the inclusion $\iota(\rspec{\mathcal I}) \subseteq \rspec[s]{\mathcal I}$ for $r<s$, 
the reader can check that $(\vartheta_{t'}\circ \iota\circ\varsigma^{t})\circ(\vartheta_{t}\circ\varsigma^{r})=(\vartheta_{t'}\circ \iota\circ\varsigma^{r})$ for all $r<t < t'\leq s$. The claim readily follows.
\end{proof}

\begin{remark}
An easy way to find pairs $\b,\d$ satisfying the axioms is to start from the functors $\cfT(\mhyphen)$ and $\cfS(\mhyphen)$ on positive comonoids $\p,\q$. We turn to that discussion next, and find a nice extension of a result appearing in \cite{aguiar2010monoidal}.
\end{remark}

\subsection{Generalizing and interpolating the abelianization map}
\label{sec:main-results}

Recall the universal property of the $\cfT(\p)$ functor from Section \ref{sec:T and S}. In case $\h = \cfS(\p)$, $\hat\zeta$ is the \emph{abelianization map} of \cite[\S11.6.2]{aguiar2010monoidal}. In the further case that $\p = \bfE_+$, the Hopf monoids $\cfT(\p)$ and $\cfS(\p)$ are $\bfSigma$ and $\bfPi$, respectively. These will produce in the coming section all the Hopf algebras mentioned in the introduction. 

Here, starting from any map $\theta:\p \onto \q$ of positive comonoids, we use the canonical decompositions for $\p$ and $\q$ to provide increasing filtrations on the Hopf monoids $\b=\cfT({\p})$ and $\d=\cfS({\q})$. The role of $\rc$ from Definition \ref{def:rc} is played here by $\cfT(\p_{\geq r})\sdot\cfS(\q_{<r})$.\footnote{An alternate proof is available. From $\theta$ we get a (positive) comonoid map $\p \to \rc_+$. Then Theorem \ref{th:T(p)} can produce all of the Hopf morphisms necessary to make \eqref{eq:ports-from-b} commute. But this approach is unavailable for some of the examples in Section \ref{sec:fock-functors}, so we follow the road map of Theorem \ref{th:axiomatic} instead.}%

\begin{theorem}\label{th:abelianization}
    Let $\theta:\p\onto\q$ be a surjective map of positive comonoids, and $\pi_\theta:\cfT(\p)\onto \cfS(\q)$ its natural extension (a Hopf map). Then $\bigl\{\cfT({\p}_{\geq r})\sdot\cfS({\q}_{<r})\bigr\}_{r\geq1}$ is a one-parameter family of Hopf monoids interpolating $\pi_\theta$.
\end{theorem}
\begin{proof} From the universal property of $\cfT(\mhyphen)$, $\pi_\theta$ is defined on simple tensors by
\[
	(p_A)_{A\in F} \mapsto \{\theta(p_A)\}_{A\in F}.
\]
Because we use the canonical decomposition of $\p$ and $\q$ to define our filtrations, we get immediately that $\downspec{\b}=\cfT(\p_{<r})$ is a Hopf submonoid of $\b$; and likewise for $\d$. Moreover, $\pi(\downspec{\b}) \subseteq \downspec{\d}$.

Since $\b$ is free as a monoid, $\langle\downspec{\b}^+\rangle$ is spanned by elements $(p_A)_{A\in F}$ with at least one $|A|<r$. So $\upspec{\b}$ is spanned by similar expressions with the stipulation that all $|A|\geq r$.
 Clearly $\upspec{\b} \supseteq \upspec[s]{\b}$ when $r<s$.
 
 Thus Axioms \ref{ax:comm}--\ref{ax:comp} are satisfied. Since $\b,\d$ are also connected, Axiom \ref{ax:stab} holds as well. This completes the proof.
\end{proof}

For concreteness, we describe the map $\port^r_s$ in this context. 
Consider a simple tensor $(p_A)_{A\in F}\otimes \{q_B\}_{B\in Y} \in \rc$, which satisfies $|A|\geq r$ and $|B|<r$ ($\forall A,\forall B$). Recall $F$ is a linear order on a set partition $X$. Write $X = X'\sqcup X''$, with $A\in X' \iff |A|\geq s$. Let $F'$ be the ordering of $X'$ induced by $X$. Then
\begin{gather}\label{eq:TS-port}
    (p_A)_{A\in F}\otimes\{q_B\}_{B\in Y} \ \xrightarrow{\,\port^r_s\,} \
    (p_A)_{A\in F'}\otimes\bigl\{\{q_B\}_{B\in Y}\cup \{\theta(p_A)\}_{A\in X''}\bigr\}.
\end{gather}
We call elements of $\upspec{\b}$ and $\downspec{\b}$ \demph{$r$-large} and \demph{$r$-small}, respectively. (So in the present context, $(p_A)_{A\in F}$ is $r$-large iff $|A|\geq r$ for all $A\in F$.) Given any $\phi = (p_A)_{A\in F}$ in $\cfT(\p)$, write $\phi^{\lfloor r \rfloor}$ and $\phi^{\lceil r \rceil}$ for the restriction of $\phi$ to its $r$-large and $r$-small $\p$-structures, respectively. Then \eqref{eq:TS-port} may be rewritten as
\begin{gather}\label{eq:TS-port-alternate}
    \phi\otimes \psi \xmapsto{\,\port^r_s\,} \phi^{\lfloor s \rfloor}\otimes \psi\cdot\pi_\theta(\phi^{\lceil s \rceil}).
\end{gather}

Finally, we remark that we could just as well have started from $\b=\cfS(\p)$. Mimicking the proof of Theorem \ref{th:abelianization}, using the the universal property of $\cfS(\mhyphen)$, one finds that $\upspec{\b} = \cfS(\p_{\geq r})$ is a complement to $\langle \downspec{\b}^+\rangle$. 

\begin{corollary}\label{th:near-abelianization}
Let $\theta:\p\onto\q$ be a surjective map of positive comonoids, and $\pi_\theta:\cfS(\p) \onto \cfS(\q)$ its natural extension (a Hopf map). Then
$\bigl\{\cfS({\p}_{\geq r})\sdot\cfS({\q}_{<r})\bigr\}_{r\geq1}$ is a one-parameter family of Hopf monoids interpolating $\pi_\theta$. \hfill \qed
\end{corollary}

\subsubsection{Enumeration}
\label{sec:r-enum}

Let $\p$ and $\q$ be positive species with exponential generating functions $\p(x)$ and $\q(x)$, respectively.
Taking the filtrations coming from the canonical filtrations, we enumerate the species $\cfT(\upspec{\p})\sdot\cfS(\downspec{\q})$:
\[
    \cfT(\upspec{\p})\sdot\cfS(\downspec{\q})(x) \ = \ \frac{1}{1-\upspec{\p}(x)}\exp\left(\downspec{\q}(x)\right) \ = \ \frac{1}{1-\sum_{k\geq r}p_k\frac{x^k}{k!}}\exp{\biggl(\sum_{k<r}q_k\frac{x^k}{k!}\biggr)}.
\]
Similarly,
\[
\cfS(\upspec{\p})\sdot\cfS(\downspec{\q})(x) \ = \ \exp\biggl(\sum_{k\geq r}p_k\frac{x^k}{k!}\biggr)\exp{\biggl(\sum_{k<r}q_k\frac{x^k}{k!}\biggr)}.
\]
\section{From Hopf Monoids to Hopf Algebras}
\label{sec:fock-functors}

In this section, we give several concrete examples of the general framework developed in Section \ref{sec:interpolation}. For each, we use the so-called Fock functors to produce interpolating families of combinatorial Hopf algebras (pointing to relevant literature, where possible).
Notably, the first example (Section \ref{sec:r-comps}) recovers the Hopf algebras discussed in the introduction.

\medskip
Let $(\gvec,\otimes)$ be the category of graded vector spaces. Consider the \demph{Fock functors}
\(
\mathcal{K},\overline{\mathcal{K}},\mathcal{K}^\vee:\Spec\to \gvec
\)
defined by
\[
\mathcal{K}(\b) := \bigoplus_{n\geq0}\b[n],\ \ 
\overline{\mathcal{K}}(\b) := \bigoplus_{n\geq0}\b[n]_{\mathfrak{S}_n}, 
\ \ \text{and } \ \ \mathcal{K}^\vee(\b) := \bigoplus_{n\geq0}\b[n],
\]
where $\b[n]_{\mathfrak{S}_n}$ is the vector space of $\mathfrak{S}_n$-coinvariants of $\b[n]$. 
The last is the contragradient of the first. That is, $\Kvee(\b^*) \cong \K(\b)^*$.
An in-depth discussion of Fock functors can be found in \cite[Part III]{aguiar2010monoidal}.\footnote{We omit $\Kbv$ since over a field of characteristic zero, $\Kbv\!(\mhyphen)\cong \Kbar(\mhyphen)$.} For us, their salient feature is that they are \emph{bilax} functors, meaning they take (a) bimonoids to bialgebras and (b) bimonoid maps to bialgebra maps. (Indeed, connected Hopf monoids in $(\Spec,\sdot)$ are mapped to Hopf algebras in $(\gvec,\otimes)$, {\it cf.} Remark \ref{rem:hopf-emptyset}.)

A brief word about the enumeration of graded dimensions for the concrete examples we give below. The formulas in Section \ref{sec:r-enum} readily enumerate $\K(\rc)$ and $\Kvee(\rc)$ ({\it cf.} the  definitions of $\K$ and $\Kvee$ above). The story for $\Kbar(\rc)$ is more intricate (involving enumeration of unlabeled structures). Since enumeration is not the goal of this section, we refer the reader to \cite{bergeron1998combinatorial} or \cite{bergeron2008introduction} for all details.

\subsection{\texorpdfstring{$r$}{r}-Compositions}
\label{sec:r-comps}

Recall that $\bfSigma  =\cfT ( \bfE_+)$ and $\bfPi = \cfS ( \bfE_+)$. The abelianization map $\pi:\bfSigma\onto\bfPi$ and Theorem \ref{th:abelianization} gives us the interpolating family $\rspec{\bfSigma} = \upspec{\bfSigma}\sdot\downspec{\bfPi}$, with basis running over pairs $(\Phi,\phi)$ of $r$-large set compositions and $r$-small set partitions, respectively (on disjoint sets $S$ and $T$, as usual). Since we take the identity map on $\bfE_+$ for $\theta$ in Theorem \ref{th:abelianization}, the carrying maps take the form
\begin{align*}
\port_s^r:\rspec{\bfSigma} &\onto \rspec[s]{\bfSigma}
\\
(\Phi;\phi) &\mapsto \bigl(\Phi^{\lfloor s\rfloor};\pi(\Phi^{\lceil s\rceil})\cup\phi\bigr).
\end{align*}
Notation as in \eqref{eq:TS-port-alternate}.

\subsubsection{Hopf algebras of $r$-compositions}
The Fock functors give us the following:
\begin{center}
\begin{tikzcd}
    \K(\bfSigma) = \mathsf{M}\Pi \arrow[d, two heads] & \Kbar(\bfSigma) = \nsym  \arrow[d, two heads] & \Kvee(\bfSigma) = \ncqsym^* \arrow[d, two heads]  \\
    \K(\bfPi) = \ncsym & \Kbar(\bfPi) = \sym  & \Kvee(\bfPi) = \ncsym^*
\end{tikzcd}
\end{center}
and dually, $\ncsym^*\hookrightarrow\mathsf{M}\Pi^*$, $\sym\hookrightarrow\qsym$, and $\ncsym\hookrightarrow\ncqsym$.
So each of the six Hopf pairs has an interpolating family. (Including two families not previously appearing in the literature.) This completes the task laid out in the introduction: the maps $\{\sym \into \rqsym \into \qsym\}_{r\geq1}$, $\{\ncsym \into \rncqsym \into \ncqsym\}_{r\geq1}$, and their duals are the ones first defined by Hivert in \cite{hivert2004local}.

\begin{remark}
    The notation for $\K(\bfSigma)$ is taken from \cite[\S17.3.1]{aguiar2010monoidal}, this Hopf algebra was first introduced in the context of twisted bialgebras (essentially bimonoids in species) by Patras and Schocker \cite{patras2008trees}. 
    The notation for $\Kvee(\bfSigma)$ is taken from \cite{bergeron2009hopf}. After the work of Novelli--Thibon and Foissy \cite{novelli2006polynomial,foissy2007bidendriform} it is known to be a self-dual Hopf algebra.
\end{remark}

\subsection{\texorpdfstring{$r$}{r}-Bijections}
\label{sec:r-bij}

We turn to the Grossman--Larson Hopf algebra of heap-ordered trees  \cite{grossman1989hopf}. We use its realization as bijective endofunctions \cite{hivert2008commutative}.

\subsubsection{Hopf monoid of bijections}

Every bijection is a collection of cycles, $\bij=\bfE\scirc \cyc$, so we look for a positive comonoid structure on the species of cycles and then define the Hopf monoid $\bij$ as $\cfT(\cyc)$. We could take the trivial structure, as in \cite{aguiar2010monoidal}, but we choose to follow Marberg \cite{marberg2015linearization}.

Given $I=S\sqcup T$ with $S,T\neq \emptyset$, the coproduct on $\cyc$ is defined by restriction
\[
 \Delta_{S,T}^{\cyc}(\gamma) = \gamma|_S\otimes \gamma|_T.
 \]
Here $\gamma\vert_{S}$ represents the unique cycle on $S$ induced by $\gamma$; and likewise for $\gamma\vert_T$. One checks that the linear map $\theta:\cyc\to\bfE_+$ sending $\gamma \in \cyc[I]$ to $\ast_I$ is in fact a map of positive comonoids. Then Corollary \ref{th:near-abelianization} applies, giving us a family $\rspec{\bij}:=\cfS(\cyc_{\geq r})\sdot\cfS(\bfE_{<r})$ interpolating the Hopf map $\pi_\theta$. (Here $\pi_\theta:\bij \onto \bfPi$ maps a collection of cycles $\gamma_i$ on sets $S_i$ to the collection of sets $S_i$.) The carrying maps take the form
\begin{align*}
    \port_s^r:\rspec{\bij} &\onto \rspec[s]{\bij}
    \\
    (\sigma;\phi) &\mapsto \bigl(\sigma^{\lfloor s\rfloor};\pi_\theta(\sigma^{\lceil s\rceil})\cup\phi\bigr).
\end{align*}
Notation as in \eqref{eq:TS-port-alternate}. Concretely, $\sigma^{\lfloor s\rfloor}$ (respectively, $\sigma^{\lceil s\rceil}$) collects those cycles of $\sigma$ with length at least $s$ (respectively, less than $s$).

\subsubsection{Hopf algebras of bijections}
The Fock functors give us the following:
\begin{center}
\begin{tikzcd}
    \K(\bij) = \ssym \arrow[d, two heads] & \Kbar(\bij) = \sym  \arrow[d, two heads] & \Kvee(\bij) = \sqsym \arrow[d, two heads]  \\
    \K(\bfPi) = \ncsym & \Kbar(\bfPi) = \sym  & \Kvee(\bfPi) = \ncsym^*
\end{tikzcd}
\end{center}
and dually, $\ncsym^*\hookrightarrow\ssym^*$, $\sym\hookrightarrow\sym$, and $\ncsym\hookrightarrow\sqsym^*$. So the Fock functors and Theorem \ref{th:abelianization} provide four Hopf pairs with a nontrivial interpolating family.

\begin{remark}\label{rem:bij}
The notations $\ssym$\footnote{Not to be confused with the Malvenuto--Reutenauer Hopf algebra; denoted ``$\fqsym$'' in \cite{hivert2008commutative}.} and $\sqsym$ are taken from the work of Hivert, Novelli, and Thibon \cite{hivert2008commutative}, who show that $\ssym^*\cong \sqsym$. 
    The equality $\K(\bij) = \ssym$ is first observed by Marberg in \cite[Section 5.4]{marberg2015linearization}. 
The central map displayed above is in fact the identity map; see the discussion preceding Proposition \ref{th:Kbar(bij)}. The equality $\Kvee(\bij) = \sqsym$ is interesting because Marberg's result would suggest rather that $\Kvee(\bij^*) = \sqsym$. This coincidence is explained in the buildup to Proposition \ref{th:Kv(bij)}. 
\end{remark}

Recall the complete homogeneous symmetric functions $\{h_d\}_{d\geq1}$ form a free (commutative) generating set for $\sym$. Put $h_0=1$. Given a partition $\lambda=(\lambda_1,\ldots, \lambda_k)$, we put $h_\lambda := h_{\lambda_1}h_{\lambda_2}\cdots h_{\lambda_k}$ and 
\[
\Delta(h_\lambda) = \prod_{1\leq i \leq k}\Delta(h_{\lambda_i})
\quad\hbox{with}\quad
\Delta(h_d) = \sum_{i+j=d} h_i \otimes h_j.
\] 

The details of the $\Kbar$ construction\footnote{Specifically, the definitions of $\overline{\varphi}_{\h,\h}, \overline{\psi}_{\h,\h}$, $\mu^{\Kbar(\h)}$, and $\Delta^{\Kbar(\h)}$ in \cite[\S\S15.1.1 \& 15.2.1]{aguiar2010monoidal}.} provide that the canonical monoid generators for both $\bij$ and $\bfPi$ map to (free, commutative) generators $\{\hbar_d\}_{d\geq1}$ satisfying
\[
\Delta(\hbar_d) = \sum_{i+j=d} \binom{d}{i} \hbar_i \otimes \hbar_j.
\]
({\it E.g.}, $\overline{(1324)} \in \Kbar(\bij)$ is identified with $\hbar_4$.) One checks that the mapping $\hbar_d \mapsto d!\,h_d$ extends to a Hopf isomorphism.

\begin{proposition}\label{th:Kbar(bij)}
    The Hopf algebras $\Kbar(\bij)$ and $\sym$ are isomorphic. \hfill\qed
\end{proposition}

In \cite{marberg2015linearization}, Marberg considers $\cfS(\p)$ for positive linearized cocommutative comonoids $\p=\Bbbk{\rfp}$. With these assumptions, the coproduct in $\p$ comes from set maps $\rfp[S\sqcup T] \rightarrow \rfp[S]\times\rfp[T]$, $p \mapsto p\vert_{S}\times p\vert_{T}$. (The notation evokes {restriction}, which is supported by the work of Schmitt \cite{schmitt1993hopf}; see \cite[Sec. 8.7]{aguiar2010monoidal}.) Marberg introduces the \demph{restriction poset} $(\rfE\circ\rfp, <)$ and uses it to build new bases for $\cfS(\p)$ (mimicking change-of-basis formulas in the theory of symmetric functions). To illustrate, we compute here the ``powersum'' basis of $\bij$ and show how the image of the Fock functors relate to some Hopf algebras in the literature.

Let $\sigma = \{\gamma_1,\ldots \gamma_i, \ldots,\gamma_k\}$ be a bijection which is a set of cycles $\gamma_i$. Then a bijection $\sigma'$ is covered by $\sigma$ if there exists an integer $i$ and sets $S$ and $T$ such that $\sigma' = \{ \gamma_1,\ldots, \gamma_{i-1}, \gamma_i|_S, \gamma_i|_T,\gamma_{i+1},\ldots,\gamma_k \}.$ Following Marberg, the powersum basis is defined, via the M\"obius function, as
\[
p_\sigma = \sum_{\sigma' \preceq \sigma } \ddot{\mu}(\sigma',\sigma) \sdot \sigma'.
\]
For example, consider the bijection $\{(ad)(bc)\}$ on the set $\{a,b,c,d\}$. Then, 
\[
p_{(ad)(bc)} = (ad)(bc) - (a)(bc)(d)-(ad)(b)(c)+(a)(b)(c)(d).
\]
See Figure \ref{fig:powersum-poset} for a larger example. Note that lower order ideals in $(\bij,<)$ are product posets. As a consequence, we have $\mu(p_{\sigma'} \otimes p_{\sigma''}) = p_{\sigma'\cup\sigma''}$ for any $\sigma' \in \mathrm{bij}[S]$,  $\sigma''\in \mathrm{bij}[T]$. The M\"obius values in the definition of $p_\sigma$ ensure that
\[
 \Delta_{S,T}(p_{\sigma}) = 
 \begin{cases}
    p_{\sigma|_S}\otimes p_{\sigma|_T} & \text{if } \sigma|_S \cup \sigma|_T  = \sigma, \\
    0 & \text{else}
\end{cases}
\]
for all $\sigma\in\mathrm{bij}[I]$.
\begin{figure}[!htb]
    \centering
\newtoggle{showmobius}
\toggletrue{showmobius}

\iftoggle{showmobius}{\colorlet{mobiuscolor}{blue}}{\colorlet{mobiuscolor}{white}}
\iftoggle{showmobius}{\colorlet{mobiuscaption}{black}}{\colorlet{mobiuscaption}{white}}

\tikzset{
  mobius/.style={font=\tiny, text=mobiuscolor, inner sep=0pt}
}

\begin{tikzpicture}[
  scale=1,
  every node/.style={inner sep=3pt},
  font=\small
]
\node (top) at (0,6) {$(adcb)$};
\node[mobius] at ([xshift=0pt, yshift=7pt]top.south west) {$1$};
\node (t5) at (-5.7,4.2)  {$(ad)(bc)$};
\node (t2) at (-3.8,4.2)  {$(a)(bdc)$};
\node (t1) at (-1.9,4.2)  {$(adc)(b)$};
\node (t3) at (0,4.2)   {$(acb)(d)$};
\node (t6) at (1.9,4.2)   {$(ab)(cd)$};
\node (t4) at (3.8,4.2)   {$(adb)(c)$};
\node (t7) at (5.7,4.2)   {$(ac)(bd)$};
\node[mobius] at ([xshift=-2pt, yshift=2pt]t5.south west) {$-1$};
\node[mobius] at ([xshift=-2pt, yshift=2pt]t2.south west) {$-1$};
\node[mobius] at ([xshift=-2pt, yshift=2pt]t1.south west) {$-1$};
\node[mobius] at ([xshift=-2pt, yshift=2pt]t3.south west) {$-1$};
\node[mobius] at ([xshift=-2pt, yshift=2pt]t6.south west) {$-1$};
\node[mobius] at ([xshift=-2pt, yshift=2pt]t4.south west) {$-1$};
\node[mobius] at ([xshift=-2pt, yshift=2pt]t7.south west) {$-1$};
\node (m4) at (-5,1.8) {$(ad)(b)(c)$};
\node (m3) at (-3,1.8) {$(a)(bc)(d)$};
\node (m1) at (-1,1.8) {$(a)(b)(cd)$};
\node (m2) at (1,1.8)  {$(a)(bd)(c)$};
\node (m6) at (3,1.8)  {$(ab)(c)(d)$};
\node (m5) at (5,1.8)  {$(ac)(b)(d)$};
\node[mobius] at ([xshift=0pt, yshift=1pt]m4.south west) {$2$};
\node[mobius] at ([xshift=0pt, yshift=1pt]m3.south west) {$2$};
\node[mobius] at ([xshift=0pt, yshift=1pt]m1.south west) {$2$};
\node[mobius] at ([xshift=0pt, yshift=1pt]m2.south west) {$2$};
\node[mobius] at ([xshift=0pt, yshift=1pt]m6.south west) {$2$};
\node[mobius] at ([xshift=0pt, yshift=1pt]m5.south west) {$2$};
\node (bot) at (0,0) {$(a)(b)(c)(d)$};
\node[mobius] at ([xshift=-2pt, yshift=-2pt]bot.south west) {$-6$};
\foreach \x in {t1,t2,t3,t4,t5,t6,t7}
  \draw (top) -- (\x);
\draw (t1) -- (m1);
\draw (t1) -- (m5);
\draw (t1) -- (m4);
\draw (t2) -- (m3);
\draw (t2) -- (m2);
\draw (t2) -- (m1);
\draw (t3) -- (m6);
\draw (t3) -- (m5);
\draw (t3) -- (m3);
\draw (t4) -- (m4);
\draw (t4) -- (m2);
\draw (t4) -- (m6);
\draw (t5) -- (m3);
\draw (t5) -- (m4);
\draw (t6) -- (m1);
\draw (t6) -- (m6);
\draw (t7) -- (m2);
\draw (t7) -- (m5);
\foreach \x in {m1,m2,m3,m4,m5,m6}
  \draw (\x) -- (bot);
\end{tikzpicture}

\caption{Elements beneath $\hat1=(adcb)$ in the restriction poset on $\mathrm{bij}[\{a,b,c,d\}]$. \textcolor{mobiuscaption}{Mobius values \textcolor{mobiuscolor}{$\ddot{\mu}(\sigma,\hat1)$} indicated in blue.}}
    \label{fig:powersum-poset}
\end{figure}
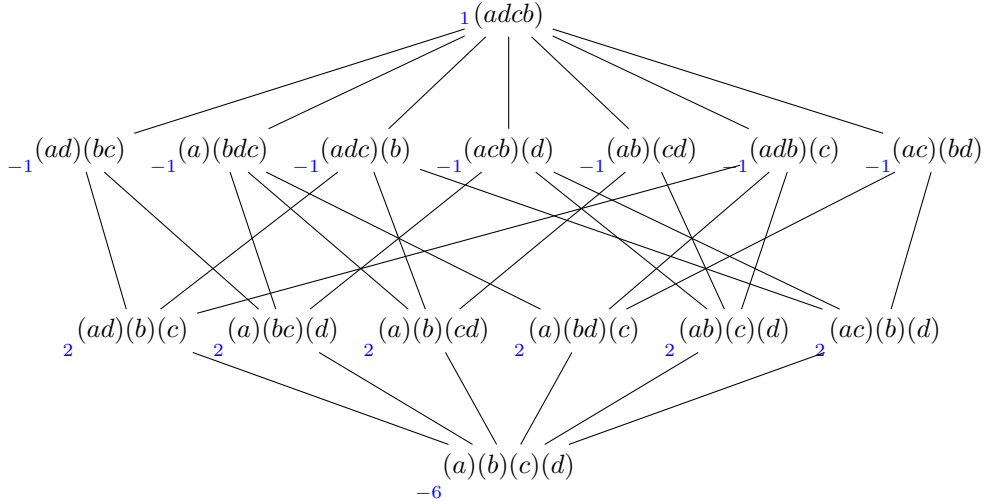

The structure of $\bij$ in the powersum basis is what allows Marberg to make his observation $\K(\bij)=\ssym$. Similarly, one can check that the Hopf structure of $\Kvee(\bij; p)$ matches that of $\sqsym$ in the $\{ M_\sigma \}$ basis from \cite{hivert2008commutative}. The curiosity mentioned in Remark \ref{rem:bij} is explained away as follows: $\bij^* \cong \bij$. This claim, not at all evident when working in the canonical basis for $\bij$, is trivial to verify in the powersum basis.

\begin{proposition}\label{th:Kv(bij)}
    The Hopf algebras $\Kvee(\bij)$ and $\sqsym$ are isomorphic. \hfill\qed
\end{proposition}

\subsection{$r$-Parking Functions}
\label{r-parking}

We turn to the Hopf algebra of parking functions defined in \cite{hivert2008commutative}. 
We begin with a definition of the set species $\mathrm{PF}$ of parking functions. (See also \cite{priez2015noncommutative}, where the authors give a recursive description and enumerative results.) 

Given a set $I$ of cardinality $n$, recall a function $\Phi:I\to [n]$ is nothing more than a list of possibly empty disjoint subsets $(A_1, A_2, \ldots, A_n)$ whose union is $I$; {\it i.e.}, a weak composition of $I$ into $n$ parts. We call $\Phi$ a \demph{parking function} if the subset sizes additionally satisfy
\[
    \sum_{i=1}^k |A_i|\geq k \hbox{ for each }1\leq k \leq n.
\]
Let $\pf = \Bbbk\mathrm{PF}$ be the corresponding vector species.

\subsubsection{Hopf monoid of parking functions}

Parking functions have a monoid structure with the product being concatenation of weak compositions:
\begin{gather*}
\mu^\pf_{S,T}: \Phi \otimes \Psi \mapsto \Phi|\Psi.
\end{gather*}
Every parking function $\Phi$ can be maximally decomposed as a concatenation $\Phi=\Phi^{(1)}|\Phi^{(2)}|\cdots|\Phi^{(k)}$ of smaller (positive) parking functions. We call the indecomposable pieces \demph{block parking functions}. Evidently, $\pf = \bfL \circ \mathbf{BPF}$ as a monoid. 
The coproduct we choose (below) does not come from a comonoid structure on block parking functions; so $\pf$ is not realized as $\cfT(\mathbf{BPF})$.

For any parking function $\Phi$ on $I$ (a weak composition), consider the induced weak composition of $S$ ($S\subseteq I$): $\Phi_{\cap{S}} = (A_1\cap S, A_2\cap S, \ldots, A_n\cap S)$. This list is not a parking function on $S$ for proper subsets $S\subsetneq I$. This deficiency is repaired with the \demph{parkization} function:
\begin{gather*}
\Delta^\pf_{S,T}: \Phi \mapsto \park(\Phi_{\cap S}) \otimes \park(\Phi_{\cap T})
\end{gather*}
We follow \cite{novelli2007hopf} and define ``$\park$'' as an algorithm (Algorithm \ref{alg:parkization}). (Briefly: given a weak composition $\Gamma$ of $S$, $\park(\Gamma)$ attempts\footnote{The algorithm succeeds if $\Gamma = \Phi_{\cap S}$ for a parking function $\Phi$, but may fail more generally.} to build a parking function $\Phi$ on $S$ by marching through the (possibly empty) blocks of $\Gamma$ and keeping those that do not violate the parking condition.) 
Table \ref{tbl:parkization} illustrates $\park(\Gamma)$ for a weak composition $\Gamma$ of $\{a,b,c,d,e\}$. 
\begin{algorithm}
\caption{Definition of $\park(\Gamma)$}
\begin{algorithmic}
\State{Initialize $\Phi := ()$; \ $\dot\Gamma:=\Gamma$}
\State{\textcolor{black!50}{\# each block of $\dot\Gamma$ shall be moved to $\Phi$ or discarded}}
\While{$\dot\Gamma \neq()$}:
\If{$|\Phi|+|\dot\Gamma_1|\geq \ell(\Phi)+1$}:
\State{$\Phi := \Phi|\dot\Gamma_1; \ \dot\Gamma := (\dot\Gamma_2,\dot\Gamma_3,\ldots)$}
\Else{}:
\State{$\dot\Gamma := (\dot\Gamma_2,\dot\Gamma_3,\ldots)$}
\EndIf
\EndWhile
\State{\Return{$\Phi$}}
\end{algorithmic} \label{alg:parkization}
\end{algorithm} 
\begin{table}\small
\begin{align*}
\Phi &= (){\color{gray!40}\,,a\,?\,\cmark} &
\dot\Gamma &= ({\color{black}\underline{a}},\emptyset,bcd,\emptyset,e,\emptyset,\emptyset)
    \\
\Phi &= (a){\color{gray!40}\,,\emptyset\,?\,\xmark} &
\dot\Gamma &= ({\color{black}\underline{\emptyset}},bcd,\emptyset,e,\emptyset,\emptyset)
    \\
\Phi &= (a){\color{gray!40}\,,bcd\,?\,\cmark} &
\dot\Gamma &= ({\color{black}\underline{bcd}},\emptyset,e,\emptyset,\emptyset)
    \\
\Phi &= (a,bcd){\color{gray!40}\,,\emptyset\,?\,\cmark} &
\dot\Gamma &= ({\color{black}\underline{\emptyset}},e,\emptyset,\emptyset)
    \\
\Phi &= (a,bcd,\emptyset){\color{gray!40}\,,e\,?\,\cmark} &
\dot\Gamma &= ({\color{black}\underline{e}},\emptyset,\emptyset)
    \\
\Phi &= (a,bcd,\emptyset,e){\color{gray!40}\,,\emptyset\,?\,\cmark} &
\dot\Gamma &= ({\color{black}\underline{\emptyset}},\emptyset)
    \\
\Phi &= (a,bcd,\emptyset,e,\emptyset){\color{gray!40}\,,\emptyset\,?\,\xmark} &
\dot\Gamma &= ({\color{black}\underline{\emptyset}})
    \\
\Phi &= (a,bcd,\emptyset,e,\emptyset) &
\dot\Gamma &= ()
\end{align*}
\caption{Parkization of $\Gamma = (a, \emptyset, bcd, \emptyset, e, \emptyset, \emptyset)$.}    
\label{tbl:parkization}
\end{table}

We need two lemmas, which follow easily from the definition of parkization. Some notation first: $\Gamma=(B_1,\ldots, B_k)$ is any weak set composition, we write $\ell(\Gamma)$ for its \emph{length}, $k$, and $|\Gamma|$ for its \emph{size}, $\sum_i |B_i|$.

\begin{lemma}[Adding empty sets] Write $\emptyset^k$ for the $k$-term list $(\emptyset,\ldots,\emptyset$). Suppose $\Gamma=\Gamma^{(1)}|\Gamma^{(2)}$ is a weak set composition satisfying: (i) $\ell(\Gamma) \geq |\Gamma|$ and (ii) $\ell(\Gamma^{(1)}) = |\Gamma^{(1)}|$. Then for all $k\geq0$, we have
\[
    \park(\Gamma) = \park(\emptyset^k|\Gamma)
    =
    \park(\Gamma^{(1)}|\emptyset^k|\Gamma^{(2)})
    = 
    \park(\Gamma|\emptyset^k).
\]
\end{lemma}

\begin{lemma}[Taking subsets] Given $\Gamma=(B_1,\ldots,B_n)$ and $\Gamma' = (A_1,\ldots,A_n)$ with $A_i\subseteq B_i$ $\forall i$. If $B_i$ is deleted in $\park(\Gamma)$  (maintaining the original numbering of the blocks) then $A_i$ is deleted in $\park(\Gamma')$ (possibly at a later step in the algorithm).
\end{lemma}

Armed with these preparatory lemmas, the following formulas are straightforward to verify. They yield, respectively, coassociativity of $\Delta^{\pf}$ and compatibility between $\mu^{\pf}$ and $\Delta^{\pf}$.

\begin{enumerate}
\item Fix $I=R\sqcup S\sqcup T$ as in \eqref{eq:coassociative}. Given $\Phi\in\mathrm{\pf}[I]$, 
\[
    \park(\park(\Phi_{\cap(S\sqcup T)})_{\cap S}) = \park(\Phi_{\cap S}).
\]
\item Fix $I = S\sqcup T$ and $S'=A\sqcup C$ as in \eqref{eq:compatible-delta}. Given $\Phi\in\mathrm{\pf}[S]$ and $\Psi\in\mathrm{\pf}[T]$, 
\[
    \park((\Phi|\Psi)_{\cap S'}) = \park(\Phi_{\cap A})|\park(\Psi_{\cap C}).
\]
\end{enumerate}

\begin{theorem}\label{th:PF-is-hopf}
The species $\pf$ is a Hopf monoid under $\mu^{\pf}$ and $\Delta^{\pf}$. \qed
\end{theorem}

\subsubsection{An interpolation of Hopf monoids}

We use Theorem \ref{th:axiomatic} to define an interpolating family $\{\rspec{\pf}\}_{r\geq1}$ between $\pf$ and $\bfPi$.

To begin, consider the morphism $\pi: \pf \to \bfPi$ defined on parking functions as follows: view $\Phi=(A_1,\ldots, A_n)$ as a weak set composition and send it to the corresponding set partition $\phi$ . (That is, $\phi = \{A_1, \ldots, A_n\}$, with empty sets dropped.) It is easy to check that $\pi$ respects product and coproduct, {\it i.e.}, is a bimonoid morphism. We indicate the relevant diagrams below but leave the work to the reader.
\[
\begin{tikzcd}
    \Phi \otimes \Psi \arrow[r, "\mu^{\pf}"]\arrow[dd, "\pi \otimes \pi"'] & \Phi|\Psi \arrow[d, "\pi"]\\[-.5ex] 
    & \substack{{\textstyle \pi(\Phi\vert\Psi)}\\[.75ex] =} \\[-5ex]
    \phi \otimes \psi \arrow[r, "\mu^{\bfPi}"'] & \phi \sqcup\psi
\end{tikzcd}
\quad\qand\quad
\begin{tikzcd}
    \Phi \arrow[r, "\Delta^{\pf}"] \arrow[dd, "\pi"'] & \park(\Phi_{\cap S})\otimes \park(\Phi_{\cap T}) \arrow[d, "\pi\otimes\pi"] \\[-.5ex] 
    & \substack{{\textstyle \pi(\Phi_{\cap S})\otimes \pi(\Phi_{\cap T})}\\[.75ex] =} \\[-5ex]
    \phi \arrow[r, "\Delta^{\bfPi}"'] & \phi\vert_{S} \otimes \phi\vert_{T}
\end{tikzcd}
\]

Define a comonoid filtration $\bfPi=\bigcup_r\downspec{\bfPi}$ as in the previous examples (using the canonical decomposition of $\bfE_+$). For $\downspec{\pf}$, we have two natural choices: (a) all blocks $A_i$ in some parking function $\Phi$ satisfy $|A_i|<r$; or (b) writing $\Phi$ in terms of block parking functions  $\Phi^{(1)}|\Phi^{(2)}|\cdots|\Phi^{(k)}$, each $\Phi^{(i)}\in\pf[S_i]$ satisfies $|S_i|<r$.
Each give a comonoid filtration
for $\pf$ that $\pi$ respects. We choose the latter in what follows, as it's easier to describe the complements $\upspec{\pf}$ from Axiom \ref{ax:comp}.

Since $\pf$ is a free monoid on $\mathbf{BPF}$, a basis for the ideal $\langle \downspec{\pf}^+\rangle$ consists of all lists 
$\Phi^{(1)}|\Phi^{(2)}|\cdots|\Phi^{(k)}$ of block parking functions where at least one set $S_i$ satisfies $|S_i|<r$. So 
\[
    \upspec{\pf} = \bfL\scirc\mathbf{BPF}_{\geq r}
    \qand 
    \rspec{\pf} = \upspec{\pf}\sdot\downspec{\bfPi}.
\]
A basis for the latter is indexed by pairs $(\Phi; \phi)$ with $\Phi = \Phi^{(1)}|\Phi^{(2)}|\cdots|\Phi^{(k)}$ being a list of $r$-large block parking functions and $\phi$ being a set of $r$-small subsets. (Concretely, we say $\Phi^{(i)}$ is $r$-large if $|S_i|\geq r$ and $r$-small otherwise. Likewise for blocks of set partitions.) For $r<s$, the carrying map is given by
\begin{align*}
    \port_s^r:\rspec{\pf} &\onto \rspec[s]{\pf} \\
    (\Phi;\phi) &\mapsto \bigl(\Phi^{\lfloor s\rfloor}; \pi(\Phi^{\lceil s\rceil})\cup \phi\bigr),
\end{align*}
where $\Phi^{\lfloor s\rfloor}$ and $\Phi^{\lceil s\rceil}$ are the subsequences of $\Phi$ consisting of $s$-large and $s$-small block parking functions, respectively.

\subsubsection{Hopf algebras of parking functions}

The map $\pi:\pf\onto\bfPi$ gives us:
\begin{center}
\begin{tikzcd}
    \K(\pf) = \pfsym \arrow[d, two heads] & \Kbar(\pf) = \cqsym^*  \arrow[d, two heads] & \Kvee(\pf) = \pqsym \arrow[d, two heads]  \\
    \K(\bfPi) = \ncsym & \Kbar(\bfPi) = \sym  & \Kvee(\bfPi) = \ncsym^*
\end{tikzcd}
\end{center}
and dually, $\ncsym^*\hookrightarrow\pfsym^*$, $\sym\hookrightarrow\cqsym$, and $\ncsym\hookrightarrow\pqsym^*$. So each of the six Hopf pairs has an interpolating family provided by Theorem \ref{th:abelianization}.

\begin{remark}
    The Hopf algebra of parking functions $\pqsym$ comes from \cite{novelli2007hopf}. The equality $\Kvee(\pf)=\pqsym$ is explained below, as it well-illustrates the $\Kvee$ construction from \cite{aguiar2010monoidal}. The cocommutative Hopf algebra of parking functions denoted $\pfsym$ is explored in \cite{li2015monomial}; it is the dual of a Hopf algebra first introduced in \cite{hivert2008commutative}. The equality $\K(\pf)=\pfsym$ is not immediately apparent and will be discussed in the build up to Proposition \ref{th:PFSym}. Finally $\cqsym$ denotes the commutative noncocommutative Hopf algebra of Dyck paths found in \cite{hivert2008commutative}. 
    The equality $\Kbar(\pf)=\cqsym^*$ is checked as in Proposition \ref{th:PFSym} so will be omitted. 
\end{remark}

Let $\Phi\in \Kvee(\pf)$ be a parking function in the $n^\text{th}$ graded component of $\Kvee(\pf)$. Then the coproduct in $\Kvee(\pf)$ is the sum of coproducts (in $\pf$) that run over the sets $[s]$ and $[s+1,s+t]$ where $n=s+t$ followed by standardization. In the language of words, this is precisely the parkization process describing the coproduct in the fundamental basis $F$ of \cite{novelli2007hopf}. Next, consider two parking functions $\Phi$ and $\Psi$ in the $s^\text{th}$ and $t^\text{th}$ graded components of $\Kvee(\pf)$, respectively. The product $\Phi\cdot\Psi$ in $\Kvee(\pf)$ is the sum of all products $\Phi^S\cdot \Psi^T$ (in $\pf$), where: $|S|=s$, $|T|=t$, $S \sqcup T = [s+t]$, $\mathsf{cano}:[s]\to S$ is the unique order-preserving bijection, and $\Phi^S$ is the image of $\Phi$ under the corresponding isomorphism $\pf[s] \xrightarrow{\pf[\mathsf{cano}]} \pf[S]$ (and likewise for $\Psi^T$). Noting that the product (in $\pf$) is concatenation and reformulating to the language of words, the sum yields a shifted shuffle product---precisely matching the $F$-basis product in \cite{novelli2007hopf}.

\begin{proposition}\label{th:PQSym}
    The Hopf algebras $\Kvee(\pf)$ and $\pqsym$ are isomorphic.
\end{proposition}

Consider now $\K(\pf)$ and a parking function $\Phi$ in the $n^\text{th}$ graded component of $\K(\pf)$. Following the $\K$ construction in \cite[Sec. 15.2.1]{aguiar2010monoidal}, the coproduct in $\K(\pf)$ is the sum of coproducts (in $\pf$) that run over the sets $S$ and $T$ such that $S\sqcup T = [n]$ followed by standardization. This means that every term will look like $\park(\Phi|_S) \otimes \park(\Phi|_T)$ followed by standardization. This also implies that the term $\park(\Phi|_T) \otimes \park(\Phi|_S)$ is also in the sum, which implies that $\K(\pf)$ is cocommutative. The product in $\K(\pf)$ is shift concatenation, which is free. Thus $\K(\pf)$ is a free, graded, connected, cocommutative Hopf algebra whose dimension sequence is given by the number of parking functions. This same sentence applies to $\pfsym$. It so happens that dimension sequences completely determine isomorphism type in this context. (See \cite[Theorem 12]{aliniaeifard2022hopf} and \cite{andrews2025hopfalgebrasdeterminedinteger} for recent work in this direction.) 

\begin{proposition}\label{th:PFSym}
    The Hopf algebras $\K(\pf)$ and $\pfsym$ are isomorphic.
\end{proposition}

\subsection{\texorpdfstring{$r$}{r}-Rooted Forests}
\label{sec:r-rooted-forests}

Our final example is centered on Hopf algebras of rooted forests, including the Connes--Kreimer Hopf algebra $\mathcal{H}_R$ and a noncommutative variant $\mathcal{H}_{PR}$. Species variants appear in \cite{aguiar2010monoidal}. We summarize the details below. A more complete history of development can be found in \cite[\S17.5.4]{aguiar2010monoidal}.

\subsubsection{An interpolation of Hopf monoids}

Writing $\vec{\a}$ (respectively, ${\a}$) for the positive species of planar rooted trees (rooted trees), put $\prf=\bfL\scirc\vec{\a}$ and $\rf=\bfE\scirc\a$. These are given bimonoid structures in \cite{aguiar2010monoidal} that share the monoid structures of $\cfT(\vec{\a})$ and $\cfS(\a)$, respectively; but the comonoids structures are more subtle. (Specifically, they do not come from comonoid structures on $\vec{\a}$ and $\a$.)

We describe the coproduct on $\prf$ briefly and leave $\rf$ for the reader. Note that $\prf$ is a linearized species, and consider some distinguished basis element (list of planar rooted trees) $\vec f \in \prf[I]$. A subset $S\subseteq I$ is \demph{$f$-admissible} if whenever $a\in S$, every ancestor of $a$ is also in $S$. 
Given $I = S\sqcup T$, we have 
\[
\Delta_{S,T}^{\prf} (\vec f) = 
\begin{cases}
    \vec f|_S\otimes \vec f|_T & \text{if }S\text{ is }\vec f\text{-admissible}\\
    0 & \text{else,}
\end{cases}
\]
where $\vec f\vert_U$ indicates the planar rooted forest on $U\subseteq I$ induced by $\vec f$. Forgetting planar embeddings of trees gives the Hopf monoid morphism $\pi:\prf \onto \rf$.  

Build comonoid filtrations on $\prf$ and $\rf$ by putting $\downspec{\prf}:=\bfL\scirc\vec{\a}_{<r}$ and $\downspec{\rf}:=\bfE\scirc{\a}_{<r}$. Clearly $\pi(\downspec{\prf})\subseteq \downspec{\rf}$, so we may apply Theorem \ref{th:axiomatic} to build an interpolating family $\rspec{\prf}$. The first step is to understand $\upspec{\prf}$ from Axiom \ref{ax:comp}. Since $\prf$ is free as a monoid (generated by planar rooted trees), we see that the ideal $\langle \downspec{\prf}^+\rangle$ is spanned by lists of planar rooted trees where at least one tree has fewer than $r$ nodes. Conclude that $\upspec{\prf}$ is spanned by planar rooted forests $\vec f$ where every tree $a\in \vec f$ has at least $r$ nodes ({\it i.e.}, $\upspec{\prf} = \bfL\scirc\vec{\a}_{\geq r}$).

For $r<s$, the carrying map is given by
\begin{align*}
    \port_s^r:\rspec{\prf}[I]&\onto\rspec[s]{\prf}[I] \\
    (\vec{f};f) &\mapsto \bigl(\vec{f}^{\lfloor s\rfloor}; \pi(\vec{f}^{\lceil s\rceil})\cup f\bigr),
\end{align*}
with notation as in \eqref{eq:TS-port-alternate}.

\subsubsection{Hopf algebras on forests}

As described in \cite[\S17.5.4]{aguiar2010monoidal}, applying $\Kbar$ to the map $\pi:\pf\onto\bfPi$ gives us maps between $\mathcal{H}_{PR}$ and $\mathcal{H}_{R}$. The complete picture is as follows:
\begin{center}
\begin{tikzcd}
    \K(\prf)= \boxed{?} \arrow[d, two heads] & \Kbar(\prf)=\mathcal{H}_{PR} \arrow[d, two heads] & \Kvee(\prf) = \boxed{?} \arrow[d, two heads]  \\
    \K(\rf) = \mathcal H_o & \Kbar(\rf)=\mathcal{H}_{R}  & \Kvee(\rf) = \boxed{?}
\end{tikzcd}
\end{center}
and dually, {\it mutatis mutandis}. Each of the six Hopf pairs has a (new) interpolating family provided by Theorem \ref{th:abelianization}.

\begin{remark}
The Hopf algebra $\mathcal{H}_o$ of {ordered forests} is studied by Foissy and Unterberger \cite{foissy2002les,foissy2013ordered1}. 
The commutative Hopf algebra $\mathcal{H}_{R}$ of rooted forests is due to Grossman and Larson \cite{grossman1989hopf}, but was popularized Connes and Kreimer through their work on renormalization in quantum field theory \cite{connes1998hopf}. 
The noncommutative analog $\mathcal{H}_{PR}$, on planar rooted forests, was introduced independently by Foissy \cite{foissy2002les} and Holtkamp \cite{holtkamp2003comparison}. 
We were unable to find $\K(\prf)$, $\Kvee(\prf)$, and $\Kvee(\rf)$ in the literature.%
\footnote{Though $\K(\prf)$ is so close to $\mathcal{H}^{\mathcal D}_{PR}$ and $\mathcal{H}_o$ studied by Foissy and collaborators that we could attribute it to him.}
\end{remark}

A complete description of the structures of the missing Hopf algebras follows from straightforward computations ({\it cf.} \cite{aguiar2010monoidal}, Example 15.17). One detail  of note is that $\K(\rf) = \mathcal H_o$ is isomorphic neither to $\Kvee(\rf)$ nor $\Kvee(\rf)^*$. Indeed, $\mathcal{H}_o$ is noncommutative and noncocommutative, while $\Kvee(\rf)$ is commutative.

\section{Further Remarks on Interpolation of Species}
\label{sec:further-questions}

We highlight some avenues for future research.

\subsection*{Alternate interpolations}

In our framework, we stitch together two filtrations via the Cauchy product. What can be said for the other natural operations on species, {\it e.g.}, $\upspec{\b}\splus\downspec{\d}$, $\upspec{\b}\stimes\downspec{\d}$, and $\upspec{\b}\scirc\downspec{\d}$? The recoil classes of \cite{novelli2011descent} fit into this story (with species operation ``$\splus$''), but a general framework for each of these alternate schemes remains elusive. 

In a different direction, we have not considered here the cases of colored and decorated species \cite[Chs. 14, 19]{aguiar2010monoidal}. Nor have we spoken about interpolating between Fock functors. In \cite{foissy2013ordered1}, the authors establish a map $\mathcal{H}_o \rightarrow \mathcal{H}_R$ that comes, in hindsight, from the natural transformation $\K \stackrel{\nu}{\Rightarrow} \Kbar$. (See \cite[\S15.1]{aguiar2010monoidal} for details on $\nu$.) Is there a sensible notion of $\rspec[s]{\!\K}$ interpolating between $\K$ and $\Kbar$ (equivalently, between $\Kvee$ and $\Kbv$)?

\subsection*{Species extensions}

We have focused our attention on the case $\b \onto \rc \onto \d$ or dually, $\b \into \rc \into \rd$. Depending on the context, one may allow for more general factorizations. The short exact sequence $\b \into \rc \onto \d$ springs to mind and suggests the question: \emph{given two Hopf monoids in species $\b,\d$, can we classify the extensions of $\d$ by $\b$?} In the theory of finite dimensional Hopf algebras, all such extensions are \emph{cleft}\footnote{Translating to our context, a cleft extension is a deformation of the usual bimonoid structure on the Cauchy product $\b\sdot\d$ by certain (co)actions $\lambda,\rho$ between $\b$ and $\d$.} and there is a suitable cohomology theory that can be worked out in specific examples \cite{Andruskiewitsch1995extensions}. The Hopf algebras arising in combinatorics (while not finite) hew closely to the finite case, so we ask for a cohomology theory in $\bimon(\Spec,\sdot)$. 

\printbibliography

@book {aguiar2010monoidal,
    AUTHOR = {Aguiar, Marcelo and Mahajan, Swapneel},
     TITLE = {Monoidal functors, species and {H}opf algebras},
    SERIES = {CRM Monograph Series},
    VOLUME = {29},
      NOTE = {With forewords by Kenneth Brown and Stephen Chase and Andr\'{e}
              Joyal},
 PUBLISHER = {American Mathematical Society, Providence, RI},
      YEAR = {2010},
     PAGES = {lii+784},
      ISBN = {978-0-8218-4776-3},
       DOI = {10.1090/crmm/029},
}

@article{aliniaeifard2024generalized,
    AUTHOR = {Aliniaeifard, Farid and Li, Shu Xiao and van Willigenburg,
              Stephanie},
     TITLE = {Generalized chromatic functions},
   JOURNAL = {Int. Math. Res. Not. IMRN},
  FJOURNAL = {International Mathematics Research Notices. IMRN},
      YEAR = {2024},
    NUMBER = {5},
     PAGES = {4456--4500},
      ISSN = {1073-7928,1687-0247},
       DOI = {10.1093/imrn/rnad149},
}

@article {aliniaeifard2022hopf,
    AUTHOR = {Aliniaeifard, Farid and Thiem, Nathaniel},
     TITLE = {Hopf structures in the representation theory of direct
              products},
   JOURNAL = {Electron. J. Combin.},
  FJOURNAL = {Electronic Journal of Combinatorics},
    VOLUME = {29},
      YEAR = {2022},
    NUMBER = {4},
     PAGES = {Paper No. 4.39, 33},
      ISSN = {1077-8926},
       DOI = {10.37236/11259},
}

@misc{andrews2025hopfalgebrasdeterminedinteger,
      author={Nicolas Andrews and Lucas Gagnon and F\'elix G\'elinas and Eric Schlums and Mike Zabrocki},
      title={When are {H}opf algebras determined by integer sequences?}, 
      year={2025},
      eprint={2505.06941},
      archivePrefix={arXiv},
      primaryClass={math.CO},
}

@article {Andruskiewitsch1995extensions,
    AUTHOR = {Andruskiewitsch, Nicol\'{a}s and Devoto, Jorge},
     TITLE = {Extensions of {H}opf algebras},
   JOURNAL = {Algebra i Analiz},
  FJOURNAL = {Rossi\u{\i}skaya Akademiya Nauk. Algebra i Analiz},
    VOLUME = {7},
      YEAR = {1995},
    NUMBER = {1},
     PAGES = {22--61},
      ISSN = {0234-0852},
}

@book {bergeron1998combinatorial,
    AUTHOR = {Bergeron, F. and Labelle, G. and Leroux, P.},
     TITLE = {Combinatorial species and tree-like structures},
    SERIES = {Encyclopedia of Mathematics and its Applications},
    VOLUME = {67},
      NOTE = {Translated from the 1994 French original by Margaret Readdy,
              With a foreword by Gian-Carlo Rota},
 PUBLISHER = {Cambridge University Press, Cambridge},
      YEAR = {1998},
     PAGES = {xx+457},
      ISBN = {0-521-57323-8},
}

@unpublished{bergeron2008introduction,
  title={Introduction to the Theory of Species of Structures},
  author={Bergeron, Fran{\c{c}}ois and Labelle, Gilbert and Leroux, Pierre},
  year={2013},
  note={Rewriting of the first chapters of `Combinatorial species\dots' \url{https://bergeron.math.uqam.ca/publications/}},
}

@article{bergeron2009hopf,
    AUTHOR = {Bergeron, Nantel and Zabrocki, Mike},
     TITLE = {The {H}opf algebras of symmetric functions and quasi-symmetric
              functions in non-commutative variables are free and co-free},
   JOURNAL = {J. Algebra Appl.},
  FJOURNAL = {Journal of Algebra and its Applications},
    VOLUME = {8},
      YEAR = {2009},
    NUMBER = {4},
     PAGES = {581--600},
      ISSN = {0219-4988,1793-6829},
       DOI = {10.1142/S0219498809003485},
}

@article {connes1998hopf,
    AUTHOR = {Connes, Alain and Kreimer, Dirk},
     TITLE = {Hopf algebras, renormalization and noncommutative geometry},
   JOURNAL = {Comm. Math. Phys.},
  FJOURNAL = {Communications in Mathematical Physics},
    VOLUME = {199},
      YEAR = {1998},
    NUMBER = {1},
     PAGES = {203--242},
      ISSN = {0010-3616},
       DOI = {10.1007/s002200050499},
}

@article {foissy2002les,
    AUTHOR = {Foissy, Lo\"ic},
     TITLE = {Les alg\`ebres de {H}opf des arbres enracin\'{e}s d\'{e}cor\'{e}s. {II}},
   JOURNAL = {Bull. Sci. Math.},
  FJOURNAL = {Bulletin des Sciences Math\'{e}matiques},
    VOLUME = {126},
      YEAR = {2002},
    NUMBER = {4},
     PAGES = {249--288},
      ISSN = {0007-4497},
       DOI = {10.1016/S0007-4497(02)01113-2},
}

@article {foissy2007bidendriform,
    AUTHOR = {Foissy, Lo\"ic},
     TITLE = {Bidendriform bialgebras, trees, and free quasi-symmetric
              functions},
   JOURNAL = {J. Pure Appl. Algebra},
  FJOURNAL = {Journal of Pure and Applied Algebra},
    VOLUME = {209},
      YEAR = {2007},
    NUMBER = {2},
     PAGES = {439--459},
      ISSN = {0022-4049},
       DOI = {10.1016/j.jpaa.2006.06.005},
}

@article {foissy2013ordered1,
    AUTHOR = {Foissy, Lo\"ic and Unterberger, J\'er\'emie},
     TITLE = {Ordered forests, permutations, and iterated integrals},
   JOURNAL = {Int. Math. Res. Not. IMRN},
  FJOURNAL = {International Mathematics Research Notices. IMRN},
      YEAR = {2013},
    NUMBER = {4},
     PAGES = {846--885},
      ISSN = {1073-7928,1687-0247},
       DOI = {10.1093/imrn/rnr273},
}

@article{garsia2007rqsym,
    AUTHOR = {Garsia, Adriano M. and Wallach, Nolan},
     TITLE = {{$r$}-{Q}sym is free over {S}ym},
   JOURNAL = {J. Combin. Theory Ser. A},
  FJOURNAL = {Journal of Combinatorial Theory. Series A},
    VOLUME = {114},
      YEAR = {2007},
    NUMBER = {4},
     PAGES = {704--732},
      ISSN = {0097-3165,1096-0899},
       DOI = {10.1016/j.jcta.2006.08.009},
}

@article {grossman1989hopf,
    AUTHOR = {Grossman, Robert and Larson, Richard G.},
     TITLE = {Hopf-algebraic structure of families of trees},
   JOURNAL = {J. Algebra},
  FJOURNAL = {Journal of Algebra},
    VOLUME = {126},
      YEAR = {1989},
    NUMBER = {1},
     PAGES = {184--210},
      ISSN = {0021-8693},
}

@unpublished{hivert2004local,
  title={Local action of the symmetric group and generalizations of quasi-symmetric functions},
  author={Hivert, Florent},
  month={December 4},
  year={2005},
  note={preprint: \url{http://igm.univ-mlv.fr/~hivert/paper.html}},
}

@article {hivert2008commutative,
    AUTHOR = {Hivert, Florent and Novelli, Jean-Christophe and Thibon,
              Jean-Yves},
     TITLE = {Commutative combinatorial {H}opf algebras},
   JOURNAL = {J. Algebraic Combin.},
  FJOURNAL = {Journal of Algebraic Combinatorics. An International Journal},
    VOLUME = {28},
      YEAR = {2008},
    NUMBER = {1},
     PAGES = {65--95},
      ISSN = {0925-9899},
       DOI = {10.1007/s10801-007-0077-0},
}

@article {holtkamp2003comparison,
    AUTHOR = {Holtkamp, Ralf},
     TITLE = {Comparison of {H}opf algebras on trees},
   JOURNAL = {Arch. Math. (Basel)},
  FJOURNAL = {Archiv der Mathematik},
    VOLUME = {80},
      YEAR = {2003},
    NUMBER = {4},
     PAGES = {368--383},
      ISSN = {0003-889X,1420-8938},
       DOI = {10.1007/s00013-003-0796-y},
}

@article{joni1979coalgebras,
    AUTHOR = {Joni, Saj-nicole A. and Rota, Gian-Carlo},
     TITLE = {Coalgebras and bialgebras in combinatorics},
   JOURNAL = {Stud. Appl. Math.},
  FJOURNAL = {Studies in Applied Mathematics},
    VOLUME = {61},
      YEAR = {1979},
    NUMBER = {2},
     PAGES = {93--139},
      ISSN = {0022-2526},
       DOI = {10.1002/sapm197961293},
}

@article{joyal1981species,
    AUTHOR = {Joyal, Andr\'{e}},
     TITLE = {Une th\'{e}orie combinatoire des s\'{e}ries formelles},
   JOURNAL = {Adv. in Math.},
  FJOURNAL = {Advances in Mathematics},
    VOLUME = {42},
      YEAR = {1981},
    NUMBER = {1},
     PAGES = {1--82},
      ISSN = {0001-8708},
       DOI = {10.1016/0001-8708(81)90052-9},
}

@article{lazzeroni2023powersum,
    AUTHOR = {Lazzeroni, Anthony},
     TITLE = {Powersum bases in quasisymmetric functions and quasisymmetric
              functions in non-commuting variables},
   JOURNAL = {Electron. J. Combin.},
  FJOURNAL = {Electronic Journal of Combinatorics},
    VOLUME = {30},
      YEAR = {2023},
    NUMBER = {4},
     PAGES = {Paper No. 4.43, 36},
      ISSN = {1077-8926},
       DOI = {10.37236/11724},
}

@article {li2015monomial,
    AUTHOR = {Li, Teresa Xueshan},
     TITLE = {The monomial basis and the {$Q$}-basis of the {H}opf algebra
              of parking functions},
   JOURNAL = {J. Algebraic Combin.},
  FJOURNAL = {Journal of Algebraic Combinatorics. An International Journal},
    VOLUME = {42},
      YEAR = {2015},
    NUMBER = {2},
     PAGES = {473--496},
      ISSN = {0925-9899},
       DOI = {10.1007/s10801-015-0587-0},
}

@article {marberg2015linearization,
    AUTHOR = {Marberg, Eric},
     TITLE = {Strong forms of linearization for {H}opf monoids in species},
   JOURNAL = {J. Algebraic Combin.},
  FJOURNAL = {Journal of Algebraic Combinatorics. An International Journal},
    VOLUME = {42},
      YEAR = {2015},
    NUMBER = {2},
     PAGES = {391--428},
      ISSN = {0925-9899},
       DOI = {10.1007/s10801-015-0585-2},
}

@article {novelli2007hopf,
    AUTHOR = {Novelli, Jean-Christophe and Thibon, Jean-Yves},
     TITLE = {Hopf algebras and dendriform structures arising from parking
              functions},
   JOURNAL = {Fund. Math.},
  FJOURNAL = {Fundamenta Mathematicae},
    VOLUME = {193},
      YEAR = {2007},
    NUMBER = {3},
     PAGES = {189--241},
      ISSN = {0016-2736,1730-6329},
       DOI = {10.4064/fm193-3-1},
}

@incollection{novelli2006polynomial,
    AUTHOR = {Novelli, Jean-Christophe and Thibon,
              Jean-Yves},
     TITLE = {Polynomial realizations of some trialgebras},
 BOOKTITLE = {18th {I}nternational {C}onference on {F}ormal {P}ower {S}eries
              and {A}lgebraic {C}ombinatorics ({FPSAC} 2006)},
     PAGES = {243--254},
      YEAR = {2006},
}

@article{novelli2011descent,
    AUTHOR = {Novelli, Jean-Christophe and Reutenauer, Christophe and
              Thibon, Jean-Yves},
     TITLE = {Generalized descent patterns in permutations and associated
              {H}opf algebras},
   JOURNAL = {European J. Combin.},
  FJOURNAL = {European Journal of Combinatorics},
    VOLUME = {32},
      YEAR = {2011},
    NUMBER = {4},
     PAGES = {618--627},
      ISSN = {0195-6698},
       DOI = {10.1016/j.ejc.2011.01.004},
}

@article {patras2008trees,
    AUTHOR = {Patras, Fr\'ed\'eric and Schocker, Manfred},
     TITLE = {Trees, set compositions and the twisted descent algebra},
   JOURNAL = {J. Algebraic Combin.},
  FJOURNAL = {Journal of Algebraic Combinatorics. An International Journal},
    VOLUME = {28},
      YEAR = {2008},
    NUMBER = {1},
     PAGES = {3--23},
      ISSN = {0925-9899,1572-9192},
       DOI = {10.1007/s10801-006-0028-1},
}

@inproceedings {priez2015noncommutative,
    AUTHOR = {Priez, Jean-Baptiste and Virmaux, Aladin},
     TITLE = {Non-commutative {F}robenius characteristic of generalized
              parking functions: application to enumeration},
 BOOKTITLE = {Proceedings of {FPSAC} 2015},
    SERIES = {Discrete Math. Theor. Comput. Sci. Proc.},
     PAGES = {733--744},
 PUBLISHER = {Assoc. Discrete Math. Theor. Comput. Sci., Nancy},
      YEAR = {2015},
}

@article{schmitt1993hopf,
    AUTHOR = {Schmitt, William R.},
     TITLE = {Hopf algebras of combinatorial structures},
   JOURNAL = {Canad. J. Math.},
  FJOURNAL = {Canadian Journal of Mathematics. Journal Canadien de
              Math\'ematiques},
    VOLUME = {45},
      YEAR = {1993},
    NUMBER = {2},
     PAGES = {412--428},
      ISSN = {0008-414X,1496-4279},
       DOI = {10.4153/CJM-1993-021-5},
}

\end{document}